\documentclass[preprint,12pt]{elsarticle}
\usepackage{graphicx}
\usepackage[cp1251]{inputenc}
\usepackage{amssymb}
\usepackage{amsthm}
\usepackage{cmap}

\usepackage{amsmath}% http://ctan.org/pkg/amsmath

\biboptions{sort&compress}

\newcommand{\sysn}{\left\{\begin{array}{rcl}}
\newcommand{\sysk}{\end{array}\right.}

\newtheorem{theorem}{Theorem}[section]
\newtheorem{lemma}[theorem]{Lemma}

\theoremstyle{example}
\newtheorem{example}[theorem]{Example}

\newtheorem{proposition}[theorem]{Proposition}
\theoremstyle{definition}
\newtheorem{definition}[theorem]{Definition}
\newtheorem{remark}[theorem]{Remark}
\newtheorem{corollary}[theorem]{Corollary}

\journal{...}

\begin{document}

\title{Baire property of space of Baire-one functions}

\author{Alexander V. Osipov}

\address{Krasovskii Institute of Mathematics and Mechanics, \\ Ural Federal
 University, Ural State University of Economics, Yekaterinburg, Russia}

\ead{OAB@list.ru}

\begin{abstract} A topological space $X$ is {\it Baire} if the Baire
Category Theorem holds for $X$, i.e., the intersection of any
sequence of open dense subsets of $X$ is dense in $X$. One of the
interesting problems for the space $B_1(X)$ of all Baire-one
real-valued functions is the Banakh-Gabriyelyan problem of
characterization of a topological space $X$ for which the function
space $B_1(X)$ is Baire.

In this paper, we solve this problem, namely, we have obtained a
characterization when a function space $B_1(X)$ is Baire for a
topological space $X$.  Also we proved that $B_1(X)$ is Baire for
any $\gamma$-space $X$ and we have obtained a characterization of
a topological space $X$ for which the function space $B_1(X)$ is a
Choquet space. This answers questions posed recently by T. Banakh
and S. Gabriyelyan.

We also conclude that, it is consistent there are no uncountable
separable metrizable space $X$ such that $B_1(X)$ is countable
dense homogeneous.

\end{abstract}
%\tnotetext[label1]{The research has been supported by .}

\begin{keyword} Baire-one function  \sep  $\gamma$-space  \sep  function space  \sep  pseudocomplete  \sep  Baire property  \sep  topological game  \sep  Choquet space  \sep  countable dense
homogeneity

\MSC[2010] 54C08 \sep 54H05 \sep 26A21

\end{keyword}

\maketitle %%
% \linenumbers

\section{Introduction}

 A space is {\it meager} (or {\it of the first Baire category}) if it
can be written as a countable union of closed sets with empty
interior. A topological space $X$ is {\it Baire} if the Baire
Category Theorem holds for $X$, i.e., the intersection of any
sequence of open dense subsets of $X$ is dense in $X$. Clearly, if
$X$ is a Baire space, then $X$ is not meager. The reverse
implication is in general not true. However, it holds for every
homogeneous space $X$ (see Theorem 2.3 in \protect\cite{LM}).
 Being a Baire space is an important topological
property for a space and it is therefore natural to ask when
function spaces are Baire. The Baire property for continuous
mappings was first considered in \protect\cite{Vid}. Then a paper
\protect\cite{LM} appeared, where various aspects of this topic
were considered. In \protect\cite{LM}, necessary and, in some
cases, sufficient conditions on a space $X$ were obtained under
which the space $C_p(X)$ of all continuous real-valued functions
$C(X)$ on a space
 $X$ with the topology of pointwise convergence is Baire.

In general, it is not an easy task to characterize when a function
space has the Baire property. The problem for $C_p(X)$ was solved
independently by Pytkeev \protect\cite{pyt1}, Tkachuk
\protect\cite{TKA} and van Douwen \protect\cite{vD}.

\medskip

\begin{theorem}(Pytkeev-Tkachuk-van Douwen)  The space
$C_p(X)$ is Baire if and only if every pairwise disjoint
sequence of non-empty finite subsets of $X$ has a strongly
discrete subsequence.
\end{theorem}

A collection $\mathcal{G}$ of subsets of $X$ is {\it discrete} if
each point of $X$ has a neighborhood meeting at most one element
of $\mathcal{G}$, and is {\it strongly discrete} if for each $G\in
\mathcal{G}$ there is an open superset $U_G$ of $G$ such that
$\{U_G: G\in\mathcal{G}\}$ is discrete.

\medskip
The problem for $C_k(X)$ was solved for locally compact $X$ by
Gruenhage and Ma \protect\cite{GMB}.  In paper \protect\cite{BG},
Banakh and Gabriyelyan introduced a new class of almost
$K$-analytic spaces which properly contains
$\check{C}$ech-complete spaces and $K$-analytic spaces and showed
that for an almost $K$-analytic space $X$ the following assertions
are equivalent: (i) $B_1(X)$ is a Baire space, (ii) $B_1(X)$ is a
Choquet space, and (iii) every compact subset of $X$ is scattered.

\medskip
One of the interesting problems for the space of Baire functions
is the Banakh-Gabriyelyan problem (Problem 1.1 in
\protect\cite{BG}): {\it Let $\alpha$ be a countable
ordinal. Characterize topological spaces $X$ and $Y$ for which the
function space $B_{\alpha}(X,Y)$ is Baire.}

\medskip

Note that in (\cite{BG}, Corollary 4.2), it is proved that
$B_{\alpha}(X,\mathbb{R})$ is Baire for any space $X$ and every
countable ordinal $\alpha\geq 2$.

\medskip
In this paper, we solve the Banakh-Gabriyelyan problem (in case
$Y=\mathbb{R}$), namely, we have obtained a characterization when
a function space $B_1(X)$ is Baire for a topological space $X$.

\section{Main definitions and notation}

Throughout this paper,  all spaces are assumed to be Tychonoff.
This is due to the fact that for any topological space $X$ there
exists the Tychonoff space $\tau X$ and the onto map $\tau_X:
X\rightarrow \tau X$ such that for any map $f$ of $X$ to a
Tychonoff space $Y$ there exists a map $g: \tau X\rightarrow Y$
such that $f=g\circ \tau_X$. Evidently, for any map $f:
X\rightarrow Y$ the unique map $\tau f: \tau X\rightarrow \tau Y$
is defined such that $\tau f \circ \tau_X=\tau_Y\circ f$. The
correspondence $X\rightarrow \tau X$ for any space $X$ and
$f\rightarrow \tau f$ for any map $f$ is called the Tychonoff
functor \protect\cite{Is,Ok}.

 The set of positive integers is denoted by $\mathbb{N}$ and
$\omega=\mathbb{N}\cup \{0\}$. Let $\mathbb{R}$ be the real line,
we put $\mathbb{I}=[0,1]\subset \mathbb{R}$, and let $\mathbb{Q}$
be the rational numbers. Let $f:X\rightarrow \mathbb{R}$ be a
real-valued function, then $\parallel f \parallel= \sup \{|f(x)|:
x\in X\}$, $S(g,\epsilon)=\{f: \parallel g-f
\parallel<\epsilon\}$, $B(g,\epsilon)=\{f: \parallel g-f
\parallel\leq\epsilon\}$, where $g$ is a real valued function and
$\epsilon>0$. Let $V=\{f\in \mathbb{R}^X: f(x_i)\in V_i,
i=1,...,n\}$ where $x_i\in X$, $V_i\subseteq \mathbb{R}$ are
bounded intervals for $i=1,...,n$, then $supp V=\{x_1,...,x_n\}$ ,
$diam V=\max \{diam V_i : 1\leq i \leq n \}$. The symbol $\bf{0}$
stands for the constant function to $0$. A basic open neighborhood
of $\bf{0}$ is of the form $[F, (-\epsilon, \epsilon)]=\{f\in
\mathbb{R}^X: f(F)\subset (-\epsilon, \epsilon)\}$, where $F\in
[X]^{<\omega}$ and $\epsilon>0$.

A real-valued function $f$ on a space $X$ is a {\it Baire-one
function} (or a {\it function of the first Baire class}) if $f$ is
a pointwise limit of a sequence of continuous functions on $X$.
Let $B_1(X)$ denote the space of all Baire-one real-valued
functions on a space
 $X$ with the topology of pointwise convergence.
 It is well known, the space $B_1(X)$ forms a vector space which
contains the space of all continuous functions $C(X)$ and which is
closed with respect to the uniform convergence. Moreover, $f\cdot
g$ and $\max(f,g)$ are Baire-one functions whenever $f,g\in
B_1(X)$.

 We recall that a subset of $X$ that is the
 complete preimage of zero for a certain function from~$C(X)$ is called a zero-set.
A subset $O\subseteq X$  is called  a cozero-set (or functionally
open) of $X$ if $X\setminus O$ is a zero-set of $X$. It is easy to
check that zero sets are preserved by finite unions and countable
intersections. Hence cozero sets are preserved by finite
intersections and countable unions. Countable unions of zero sets
will be denoted by $Zer_{\sigma}$ (or $Zer_{\sigma}(X)$),
countable intersection of cozero sets by $Coz_{\delta}$ (or
$Coz_{\delta}(X)$). It is easy to check that $Zer_{\sigma}$-sets
are preserved by countable unions and finite intersections. Note
that if $Y\subset X$ and $A\in Coz_{\delta}(X)$ then $A\cap Y\in
Coz_{\delta}(Y)$. Note also that any zero set is $Coz_{\delta}$
and any cozero-set is $Zer_{\sigma}$. It is well known that $f$ is
of the first Baire class if and only if $f^{-1}(U)\in
Zer_{\sigma}$ for every open $U\subseteq \mathbb{R}$ (see Exercise
3.A.1 in \protect\cite{lmz1}).

If $f$ is a real-valued function on a set $X$ and
$a\in\mathbb{R}$, we write $[f\geq a]$ for the set $\{x\in X:
f(x)\geq a\}$. Similarly we use $[f\leq a]$, $[f<a]$, $[f>a]$ and
$[f=a]$.

Below, in the main results, we need to apply the following two
propositions.

\medskip

\begin{proposition}(Proposition 3 in \protect\cite{KS})\label{pr1} If $A$ is a $Coz_{\delta}$-subset of a space
$X$, then there exists a Baire-one function $f$ with values in
$[0,1]$ such that $A=[f=0]$.

If $A$ and $B$ is a pair of disjoint $Coz_{\delta}$-subsets of
$X$, then there exists a Baire-one function $f$ on $X$ with values
in $[0,1]$ such that $A=[f=0]$ and $B=[f=1]$.

\end{proposition}

\medskip

\begin{proposition}(Proposition 8 in \protect\cite{KS})\label{pr2} Let $X$ be a
topological space and $Y\subset X$. Then the following assertions
are equivalent:

(i) For any Baire-one function $f$ on $Y$ there is a Baire-one
function $g$ on $X$ extending $f$ such that $\inf f(Y)=\inf g(X)$
and $\sup f(Y)=\sup g(X)$.

(ii) Any Baire-one function on $Y$ can be extended to a Baire-one
function on $X$.

(iii) For any $Coz_{\delta}$-subset $A$ of $Y$ there is a
$Coz_{\delta}$-subset $\widehat{A}$ of $X$ with $A=\widehat{A}\cap
Y$, and, moreover, for any $Coz_{\delta}$-subset $G$ of $X$
disjoint with $Y$ there is a $Coz_{\delta}$-set $H\subset X$
satisfying $Y\subset H\subset X\setminus G$.
\end{proposition}

\medskip
Both Baire and meager space have game characterizations due to
Oxtoby \protect\cite{Ox}.

The game $G_I(X)$ is started by the player ONE who selects a
nonempty open set $V_0\subseteq X$. Then the player TWO responds
selecting a nonempty open set $V_1\subseteq V_0$. At the $n$-th
inning the player ONE selects a nonempty open set $V_{2n}\subseteq
V_{2n-1}$ and the player TWO responds selecting a nonempty open
set $V_{2n+1}\subseteq V_{2n}$. At the end of the game, the player
ONE is declared the winner if $\bigcap\limits_{n\in \omega} V_n$
is empty. In the opposite case the player TWO wins the game
$G_I(X)$.

The game $G_{II}(X)$ differ from the game $G_I(X)$ by the order of
the players. The game $G_{II}(X)$ is started by the player TWO who
selects a nonempty open set $V_0\subseteq X$. Then player ONE
responds selecting a nonempty open set $V_1\subseteq V_0$. At the
$n$-th inning the player TWO selects a nonempty open set
$V_{2n}\subseteq V_{2n-1}$ and the player ONE responds selecting a
nonempty open set $V_{2n+1}\subseteq V_{2n}$. At the end of the
game, the player ONE is declared the winner if
$\bigcap\limits_{n\in\omega} V_n$ is empty. In the opposite case
the player TWO wins the game $G_{II}(X)$.

The following classical characterizations can be found in
\protect\cite{Ox}.

\medskip

\begin{theorem}(Oxtoby)\label{thox} A topological space $X$ is

(1) meager if and only if the player ONE has a winning strategy in
the game $G_{II}(X)$;

(2) Baire if and only if the player ONE has no winning strategy in
the game $G_{I}(X)$.

\end{theorem}

\medskip

A topological space $X$ is defined to be {\it Choquet} if the
player TWO has a winning strategy in the game $G_{I}(X)$. Choquet
spaces were introduced in 1975 by White \cite{Whi} who called them
{\it weakly $\alpha$-favorable spaces}.

\medskip

We note the following results obtained in \protect\cite{LM}.

(1). If $Y$ is a homogeneous non-meager space then $Y$ is Baire.

Thus, if $B_1(X)$ (or $C_p(X)$) is a non-meager space then it is
Baire.

(2). If $X$ is a Tychonoff space, $C_p(X)$ is Baire and
$Y\subseteq X$, then $C_p(Y)$ is Baire.

Note also that in \protect\cite{Tk}, a game criterion is given for
the fact that a set $Z\subseteq Y^X$ is second category, where $Y$
is a complete metric space.

 In \protect\cite{LM}, there is a game for a topological space $X$, denoted by $\Gamma(X)$.
The game $\Gamma(X)$ is a game in which two players, called ONE
and TWO, are given an arbitrary finite starting set $S_0$ and then
proceed to choose alternate terms in a sequence $S_0$, $S_1$,
$S_2$,... of pairwise disjoint finite (possibly empty) subsets of
$X$. The resulting sequence $(S_0, S_1, S_2,...)$ is called a {\it
play} of the game $\Gamma(X)$ and is said to result in a {\it win
for player ONE} if and only if the set $S_1\cup S_3\cup S_5\cup
...$ is {\it not} a closed discrete subspace of $X$. If player ONE
does not win, then player TWO wins.

A {\it strategy} for player ONE in the game $\Gamma(X)$ is a
function $\sigma$ which assigns to each pairwise disjoint sequence
$(S_0,S_1,...,S_{2k})$, with $k\geq 0$, a finite set $S_{2k+1}$
which is disjoint from $S_0\cup S_1\cup...\cup S_{2k}$.

A strategy $\sigma$ for player ONE in the $\Gamma(X)$ is said to
be a {\it winning strategy} if, whenever $(S_0,S_1,...)$ is a play
of the game $\Gamma(X)$ in which
$S_{2k+1}=\sigma(S_0,S_1,...,S_{2k})$ for each $k\geq 0$, then
player ONE wins that play.

In (\protect\cite{LM}, Theorem 4.6), it is proved that if $C_p(X)$
is Baire then ONE no has a winning strategy in the game
$\Gamma(X)$. The converse, in general, is not the case.

In \protect\cite{pyt1}, Pytkeev proposed a modification of the
game $\Gamma(X)$ - the game $\Gamma'(X)$, in which player ONE wins
if  the set $S_1\cup S_3\cup S_5\cup ...$ is not strongly
discrete.

The following characterizations can be found in
\protect\cite{pyt1}.

\medskip

\begin{theorem}(Pytkeev). For a Tychonoff space $X$ the following assertions are
equivalent:

1. $C_p(X)$ is meager;

2. there is a pairwise disjoint sequence of non-empty finite subsets of $X$ no subsequence of which is strongly discrete;

3. ONE has a winning strategy in the game $\Gamma'(X)$;

4. there is a pairwise disjoint sequence $\{\Delta_n : n\in
\mathbb{N}\}$ of finite subsets of $X$  such that
$\sup\limits_{n\in \mathbb{N}}\Bigl(\min \{|f(x)|: x\in
\Delta_n\}\Bigr)<\infty$ for each $f\in C(X)$.
\end{theorem}

\medskip

A subset $A$ of a space $X$ is called a {\it bounded} subset (in
$X$) if every continuous real-valued function on $X$ is bounded on
$A$.

\medskip

\begin{remark}(Proposition 5.1 in \protect\cite{dkm}, Corollary 3.3 in \protect\cite{LM} for an infinite
pseudocompact subspace $A$)  If $X$ is a space containing an
infinite bounded subset $A$ (e.g. a non-trivial convergent
sequence) then $C_p(X)$ is meager.
\end{remark}

\section{Baireness for space of Baire-one functions}
A $Coz_{\delta}$-subset of $X$ containing $x$ is called a {\it
$Coz_{\delta}$ neighborhood} of $x$.

\medskip

\begin{definition}  A set $A\subseteq X$ is called {\it strongly
$Coz_{\delta}$-disjoint}, if there is a pairwise disjoint
collection $\{F_a: F_a$ is a $Coz_{\delta}$ neighborhood of $a$,
$a\in A\}$  such that $\{F_a: a\in A\}$ is a {\it completely
$Coz_{\delta}$-additive system}, i.e. $\bigcup\limits_{b\in B}
F_b\in Coz_{\delta}$ for each $B\subseteq A$.

A disjoint sequence $\{\Delta_n: n\in \mathbb{N}\}$ of (finite)
sets is  {\it strongly $Coz_{\delta}$-disjoint} if the set
$\bigcup\{\Delta_n: n\in \mathbb{N}\}$ is strongly
$Coz_{\delta}$-disjoint.

\end{definition}

\medskip

Let $Game(X)$ be also a modification of the game $\Gamma(X)$ in
which player ONE wins if  the set $S_1\cup S_3\cup S_5\cup ...$ is
{\it not} strongly $Coz_{\delta}$-disjoint. If player ONE does not
win, then player TWO wins.

\medskip

\begin{lemma}\label{lem32} Let $X$ be a Hausdorff space and $\{F_i: i\in \mathbb{N}\}$ is a disjoint completely $Coz_{\delta}$-additive
system. Then any function on $\bigcup F_i$ which is constant on
each $F_i$ can be extended to a Baire-one function on $X$.
\end{lemma}

\begin{proof} Step 1: The function is bounded. Proceed as in the proof of Proposition 7 , implication (iii)
implies (i) in \protect\cite{KS}.

Let $F=\bigcup F_i$ and $f: F\rightarrow\mathbb{R}$ such that
$f(F_i)=a_i$ for some $a_i\in \mathbb{R}$ and  $\{a_i:
i\in\mathbb{N}\}$ is bounded in $\mathbb{R}$.

Set $$ t:=\left\{
\begin{array}{lcl}
f \, \, \, \, \, \, \, \, \, \, on \, \, \, \,  F\\
\inf f(F) \, \, \, on \, \, \, X\setminus F,\\
\end{array}
\right.
$$

and

$$ s:=\left\{
\begin{array}{lcl}
f \, \, \, \, \, \, \, \, \, \, on \, \, \, \,  F\\
\sup f(F) \, \, \, on \, \, \, X\setminus F.\\
\end{array}
\right.
$$

According to Theorem 3.2 in \protect\cite{lmz1}, there exists a
Baire-one function $g$ on $X$ satisfying $t\leq g\leq s$ if and
only if the following condition is satisfied: {\it given a couple
of real numbers $a<b$, there is a Baire-one function $\varphi$ on
$X$ such that}

$\varphi=0$ on $A:=[s\leq a]$ and $\varphi=1$ on $B:=[t\geq b]$.

 Using the completely
$Coz_{\delta}$-additivity of $\{F_i: i\in \mathbb{N}\}$ it is easy
to check this condition.

So assume that $a<b$ are given. Without loss of generality we may
suppose that $\inf f(F)\leq a<b \leq \sup f(F)$.

 Then $A=[s\leq a]=\bigcup
\{F_{a_i}: a_i\leq a\}$ and $B=[t\geq b]=\bigcup\{F_{a_i}: a_i\geq
b\}$. Note that $A\cap B=\emptyset$. Since $\{F_i: i\in
\mathbb{N}\}$ is a completely $Coz_{\delta}$-additivity system,
then $A,B\in Coz_{\delta}(X)$. By Proposition \ref{pr1}, there
exists a Baire-one function $\varphi$ on $X$ with values in
$[0,1]$ such that $A=[\varphi=0]$ and $B=[\varphi=1]$. Thus there
is a Baire-one function $g$ on $X$ such that $t\leq g \leq s$.
Obviously, $g\upharpoonright F=f$.

Step 2: The function is unbounded. Proceed as in the proof of
Proposition 8, implication (iii) implies (i) in \protect\cite{KS}.

Let $f$ be a unbounded function on $F$ constant on each $F_i$ thus
there is a homeomorphism $\varphi: \mathbb{R}\rightarrow (-1,1)$
such that $\inf (\varphi\circ f)(F)<0< \sup (\varphi\circ f)(F)$.

The key thing that $\varphi\circ f$ is a bounded function on $F$
constant on each $F_i$ and hence we may use Step 1. Let $h$ be a
Baire-one function on $X$ such that $h\upharpoonright
F=\varphi\circ f$, $\sup h(X)=\sup(\varphi\circ f)(F)\leq 1$ and
$-1\leq \inf (\varphi\circ f)(F)=\inf h(X)$.

By setting $G:=h^{-1}(\{-1\}\cup \{1\})$ we obtain a
$Coz_{\delta}$-subset of $X$ which is disjoint with $F$. Since $F$
is a $Coz_{\delta}$-subset of $X$, by Proposition \ref{pr1}, there
exists a Baire-one function $\psi$ on $X$ with values in $[0,1]$
such that $G=[\psi=0]$ and $F=[\psi=1]$. One can readily verify
that $g:=\varphi^{-1}\circ (h\cdot \psi)$ is a Baire-one function
on $X$ which satisfies our requirements. This conclude the proof.

\end{proof}

\medskip

\begin{lemma}\label{lem101} Let $X$ be a Hausdorff space  and let
$\{F_i: i\in \mathbb{N}\}$ forms a disjoint completely
$Coz_{\delta}$-additive system. Then any family $\{L_i:
L_i\subseteq F_i, i\in\mathbb{N}\}$ of $Coz_{\delta}$ subsets of
$X$ is a completely $Coz_{\delta}$-additive system.

\end{lemma}

\begin{proof}   Let $F=\bigcup \{F_i: i\in \mathbb{N}\}$. Consider the function $f: F
\rightarrow \mathbb{R}$ such that $f(F_{i})=i$ for every $i\in
\mathbb{N}$. Let $V$ be an open set of $\mathbb{R}$. Since
$f^{-1}(V)=\bigcup\{F_i: i\in V\}\in \mathrm{Zer}_{\sigma}(F)$
(because $\bigcup\{F_j: j\not\in V\}\in \mathrm{Coz}_{\delta}(X)$
and, hence, $\bigcup\{F_j: j\not\in V\}\in
\mathrm{Coz}_{\delta}(F)$), $f\in B_1(F)$.

 By Lemma \ref{lem32}, there is $\widetilde{f}\in B_1(X)$
such that $\widetilde{f}\upharpoonright F=f$. Let
$S_i=\widetilde{f}^{-1}(i-\frac{1}{2},i+\frac{1}{2})$ for every
$i\in \mathbb{N}$. Then $S_i\in Zer_{\sigma}(X)$. Since $L_i\in
Coz_{\delta}(X)$ for every $i\in \mathbb{N}$, we get $S_i\setminus
L_i\in Zer_{\sigma}(X)$ for every $i\in \mathbb{N}$.

Let $\{i_k: k\in \mathbb{N}\}$ be a subsequence of $\{i:
i\in\mathbb{N}\}$. Then $D=\bigcup \{S_{i_k}\setminus L_{i_k}:
k\in \mathbb{N}\}\in Zer_{\sigma}(X)$. Since $\bigcup\{F_{i_k}:
k\in\mathbb{N}\}\in Coz_{\delta}(X)$, $\bigcup\{L_{i_k}:
k\in\mathbb{N}\}=(\bigcup\{F_{i_k}: k\in \mathbb{N}\})\setminus
D\in Coz_{\delta}(X)$. Thus, $\{L_i: i\in\mathbb{N}\}$ forms a
completely $Coz_{\delta}$-additive system.
\end{proof}

\medskip

\begin{definition} A family of sets is {\it strongly $Coz_{\delta}$-disjoint} if
each set has a superset such that the family of all the supersets is disjoint and
completely $\mathrm{Coz}_{\delta}$-additive.
\end{definition}

\medskip
\begin{proposition}\label{pr34} A pairwise disjoint sequence $\{\Delta_n: n\in \mathbb{N}\}$ of non-empty finite subsets of $X$ is
strongly $Coz_{\delta}$-disjoint in the new sense if and only if it is strongly
$Coz_{\delta}$-disjoint according to Definition 1.
\end{proposition}

\begin{proof}
Let $\{\Delta_n: n\in \mathbb{N}\}$ be a strongly $Coz_{\delta}$-disjoint family of non-empty
finite subsets of $X$. Then there is a family $\{B_n: n\in \mathbb{N}\}$ such that
$\Delta_n\subseteq B_n$ for each $n$, and it is
disjoint and completely $\mathrm{Coz}_{\delta}$-additive. Let
$h:\bigcup\limits_{n\in \mathbb{N}}B_n\rightarrow \mathbb{R}$ be a
function such that $h(B_n)=n$ for every $n\in \mathbb{N}$. Then,
by Lemma \ref{lem32}, there is $\widetilde{h}\in B_1(X)$ such that
$\widetilde{h}\upharpoonright \bigcup\limits_{n\in
\mathbb{N}}B_n=h$. Let
$S_n=\widetilde{h}^{-1}((n-\frac{1}{2},n+\frac{1}{2}))$ and
$P_n=\widetilde{h}^{-1}(n)$ for every $n\in \mathbb{N}$. Since
$\widetilde{h}\in B_1(X)$, we get that $P_n\subseteq S_n\in
\mathrm{Zer}_{\sigma}$ and $\Delta_n\subseteq B_n\subseteq
P_n\in \mathrm{Coz}_{\delta}$. Since
$\Delta_n=\{x_{1,n},x_{2,n},...,x_{k(n),n}\}$ is finite, there are
co-zero neighborhoods $W(x_{j,n})$ of $x_{j,n}$, $j=1,...,k(n)$
such that $W(x_{i,n})\cap W(x_{j,n})=\emptyset$ for $i\neq j$ and
$i,j\in \{1,...,k(n)\}$.

Let $F_{j,n}=W(x_{j,n})\cap B_n$ for $j=1,...,k(n)$ and every
$n\in \mathbb{N}$.

We claim that $\gamma=\{F_{j,n}: j=1,...,k(n)$, $n\in
\mathbb{N}\}$ is completely $\mathrm{Coz}_{\delta}$-additive.

1. Since $W(x_{j,n})$ is co-zero and $B_n\in
\mathrm{Coz}_{\delta}$ then $F_{j,n}\in
\mathrm{Coz}_{\delta}$.

2. Let $\gamma_1\subseteq \gamma$. We show that $\bigcup \gamma_1\in \mathrm{Coz}_{\delta}$. Note that $\gamma=\bigcup_{n\in \mathbb{N}} F_n$ where $F_n=\{F_{j,n}: j=1,..., k(n)\}$. For each $n$, let $\gamma_{1,n}=\gamma_1\cap F_n$. Let $M=\{n\in \mathbb{N}: \gamma_{1,n}\neq \emptyset\}$.
Then $\gamma_1=\bigcup_{n\in M} \gamma_{1,n}$. For each $n\in M$, $\bigcup \gamma_{1,n}\in \mathrm{Coz}_{\delta}$ since $\gamma_{1,n}$ is finite.
Since $S_n\in \mathrm{Zer}_{\sigma}$ for each $n$, we have $S_n\setminus (\bigcup \gamma_{1,n})\in \mathrm{Zer}_{\sigma}$ for each $n\in M$. As a result,
$\bigcup_{n\in M} \biggl(S_n\setminus (\bigcup \gamma_{1,n})\biggl)\in \mathrm{Zer}_{\sigma}$. Since $\{B_n\}$ is completely
$\mathrm{Coz}_{\delta}$-additive, $\bigcup_{n\in M}  B_{n}\in
\mathrm{Coz}_{\delta}$. Note the following equality.

 $\bigcup \gamma_1=(\bigcup_{n\in M}
B_{n})\setminus \Bigl[\bigcup_{n\in M} \Bigl(S_{n}\setminus
(\bigcup \gamma_{1,n})\Bigr)\Bigr]\in \mathrm{Coz}_{\delta}$.

Thus, $\gamma$ is completely $\mathrm{Coz}_{\delta}$-additive. It follows that $\bigcup \{\Delta_n: n\in \mathbb{N}\}$ is strongly  $\mathrm{Coz}_{\delta}$-disjoint.

The other direction is easy to check, so the proof is not given.

\end{proof}

\begin{theorem}\label{th22} For a Tychonoff space $X$ the following assertions are
equivalent:

1. $B_1(X)$ is meager;

2. there is a pairwise disjoint sequence $\{\Delta_n : n\in
\mathbb{N}\}$ of non-empty finite subsets of $X$ such that \,
$\sup\limits_{n\in \mathbb{N}}\Bigl(\min \{|f(x)|: x\in
\Delta_n\}\Bigr)<\infty$ for each $f\in B_1(X)$;

3. there is a pairwise disjoint sequence of non-empty finite subsets of $X$ no subsequence of which is strongly $Coz_{\delta}$-disjoint;

4. ONE has a winning strategy in the game $Game(X)$.

%5. ONE has a winning strategy in the game $G_{II}(B_1(X))$.
\end{theorem}

\begin{proof}

$(1)\Rightarrow(3)$. Let $B_1(X)=\bigcup_{n\in\mathbb{N}} F_n$,
where $F_n$ is nowhere dense in $B_1(X)$ and $F_n\subseteq
F_{n+1}$ for every $n\in \mathbb{N}$. Suppose for a contradiction
that every sequence $\{\Delta_n : n\in \mathbb{N}\}$ of pairwise
disjoint finite subsets of $X$ contains a strongly
$\mathrm{Coz}_{\delta}$-disjoint subsequence.

\medskip

{\it Claim 1. There are a sequence $\{\Delta_i : i\in
\mathbb{N}\}$ of pairwise disjoint finite subsets of $X$, a
sequence $\{\gamma_i: i\in \mathbb{N}\}$ of finite families of
basis open sets in $\mathbb{R}^X$, and a sequence $\{m_i: i\in
\mathbb{N}\}\subseteq \mathbb{N}$ such that the following
conditions hold for every $i\in \mathbb{N}$:

$(a')$ $1\leq m_1$ and $m_i+2\leq m_{i+1}$;

$(b')$ $\overline{U}^{\mathbb{R}^X}\cap F_i=\emptyset$ and
$supp(U)\subseteq \bigcup_{j=1}^i \Delta_j$ for every $U\in
\gamma_i$;

$(c')$ if $f\in U\in \gamma_i$, then $|f(x)|\leq m_i$ for every
$x\in supp(U)$;

$(d')$ if $\varphi: A_i:=\bigcup\limits_{j=1}^i \Delta_j
\rightarrow \mathbb{R}$ is such that $\|\varphi\|_{A_i}\leq m_i$,
then there is a continuous function $f\in \bigcup \gamma_{i+1}$
such that $\|\varphi-f\|_{A_i}<\frac{1}{i}$. }

\medskip

{\it Proof of Claim 1.} For $i=1$, choose a basic open set
$U\subseteq \mathbb{R}^X$ such that
$\overline{U}^{\mathbb{R}^X}\cap F_1=\emptyset$ (this is possible
because $F_1$ is nowhere dense). Let

$\Delta_1:=supp(U)$, $\gamma_1:=\{U\}$, and
$m_1:=\sup\{\|f\|_{\Delta_1}: f\in U\}+2$.

Assume that for $1\leq i\leq n$, we have constructed $\Delta_i$,
$\gamma_i$ and $m_i$ which satisfy $(a')-(d')$.

Consider the compact set $K=\Bigl\{\varphi: dom(\varphi)=A_n$ and $\|\varphi\|_{A_n}\leq m_n
\Bigr\}$. For each $\varphi\in K$, choose a basic open subset $O_{\varphi}$ of $\mathbb{R}^X$ such that for each $\varphi_1\in O_{\varphi}$, we have

$\|\varphi_1\upharpoonright A_n- \varphi\|_{A_n}<\frac{1}{2^{n+1}}$. Since $F_{n+1}$ is nowhere dense, each $O_{\varphi}$ can be chosen
such that $\overline{O_{\varphi}}\cap F_{n+1}=\emptyset$ (closure taken in $\mathbb{R}^X$). Since $K$ is compact, there exists $\{\psi_1, ... , \psi_k\}\subset K$ such that $K\subset O_{\psi_1}\cup ... \cup O_{\psi_k}$. Let $U_t=O_{\psi_t}$ for each $t\leq k$.

Let $\gamma_{n+1}:=\{U_1,...,U_k\}$, $\Delta_{n+1}:=\bigcup_{t=1}^k supp(U_t)\setminus \bigcup_{i=1}^n
\Delta_i$ and $m_{n+1}:=m_n+\sup\Bigl\{\|f\|_{supp(U)}: f\in U\in \gamma_i$ and
$1\leq i\leq n+1\Bigr\}+1$.

It is clear that $\{\Delta_i\}_{i=1}^{n+1}$ is pairwise disjoint
and the conditions $(a')-(c')$ are satisfied. To check $(d')$, fix
$\varphi: A_n\rightarrow \mathbb{R}$ such that
$\|\varphi\|_{A_n}\leq m_n$. Choose $1\leq t\leq k$ such that
$\|\psi_t-\varphi\|_{A_n}\leq \frac{1}{2^{n+1}}$. Since
$supp(U_t)$ is finite, $U_t$ contains continuous functions, take
an arbitrary continuous function $f\in U_t$. Then, by (1), we have

$\|\varphi-f\|_{A_n}\leq
\|\varphi-\psi_t\|_{A_n}+\|\psi_t-f\|_{A_n}\leq
\frac{2}{2^{n+1}}=\frac{1}{2^n}<\frac{1}{n}.$

The claim is proved.   \, \, \, \, \, \, \, \, \, \, \, \, \, \,
\, \, \, \, \, \, \, \, \, \, \, \, \, \, \, \, \, \, \, \, \,
$\Box$

Now we redefine the sequences in Claim 1 to make simpler and
clearer their usage in what follows.

By assumption the sequence $\{\Delta_n: n\in \mathbb{N}\}$
constructed in Claim 1 contains a strongly
$\mathrm{Coz}_{\delta}$-disjoint subsequence $\{\Delta_{n_k}: k\in
\mathbb{N}\}$, where $1<n_1<n_2<\dots$. Put

$R_1:=F_1$, $\Omega_1:=\bigcup\limits_{i=1}^{n_1-1}\Delta_i$,
$\mu_1:=\bigcup\limits_{i=1}^{n_1-1} \gamma_i$, $l_1:=m_{n_1-1}$,

and for every $k\in \mathbb{N}$, we define

$R_{2k}:=F_{n_k}$, $\Omega_{2k}:=\Delta_{n_k}$,
$\mu_{2k}:=\gamma_{n_k}$, $l_{2k}:=m_{n_k}$ and

$R_{2k+1}:=F_{n_k+1}$,
$\Omega_{2k+1}:=\bigcup\limits_{i=n_k+1}^{n_{k+1}-1} \Delta_{i}$,
$\mu_{2k+1}:=\bigcup\limits_{i=n_k+1}^{n_{k+1}-1}\gamma_{i}$, and

$l_{2k+1}:=m_{n_{k+1}-1}$.

It is clear that $\{R_i: i\in\mathbb{N}\}$ is an increasing
sequence of nowhere dense sets in $B_1(X)$ such that
$B_1(X)=\bigcup_{n\in \mathbb{N}} R_i$. \, \, \, \, \, \, \, \, \, $\Box$

\medskip

{\it Claim 2. The sequence $\{R_i: i\in
\mathbb{N}\}$, the pairwise disjoint sequence $\{\Omega_i, i\in
\mathbb{N}\}$, and the sequences $\{\mu_i: i\in \mathbb{N}\}$ and
$\{l_i: i\in\mathbb{N}\}$ satisfy the following conditions $(i\in
\mathbb{N}):$

$(a)$ $1\leq l_1$ and $l_i+2\leq l_{i+1}$;

$(b)$ $\overline{U}^{\mathbb{R}^X}\cap R_i=\emptyset$ and
$supp(U)\subseteq \bigcup\limits_{j=1}^i \Omega_j$ for every $U\in
\mu_i$;

$(c)$ if $f\in U\in\mu_i$, then $|f(x)|\leq l_i$ for every $x\in
supp(U)$;

$(d)$ if $\varphi: A_i:=\bigcup_{j=1}^i \Omega_j\rightarrow
\mathbb{R}$ is such that $\|\varphi\|_{A_i}\leq l_i$, then there
is a continuous function $f\in \bigcup \mu_{i+1}$ such that
$\|\varphi-f\|_{A_i}<\frac{1}{i}$.

Moreover, the sequence $\{\Omega_{2i}: i\in \mathbb{N}\}$ is
strongly $\mathrm{Coz}_{\delta}$-disjoint.}

\medskip

{\it Proof of Claim 2.} By construction, $\{\Omega_i:i\in
\mathbb{N}\}$ is a sequence of pairwise disjoint finite subsets of
$X$, all families $\mu_i$ are finite, and the sequence
$\{\Omega_{2i}:i\in \mathbb{N}\}$ is strongly
$\mathrm{Coz}_{\delta}$-disjoint by the choice of
$\{\Delta_{n_k}:k\in \mathbb{N}\}$. The conditions (a) and (c) are
satisfied by $(a')$ and $(c')$, respectively. Since $F_i\subseteq
F_j$ for every $i\leq j$, we have $supp(U)\subseteq
\bigcup_{j=1}^i \Omega_j$ for every $U\in \mu_i$, and hence the
condition (b) holds true. The condition (d) is satisfied by the
definition of $\mu_i$ and the condition $(d')$ for the sets
$\gamma_i$. The claim is proved. \, \, \, \, \, \, \, \, \, $\Box$

\medskip

Since, by Claim 2, the sequence $\{\Omega_{2i}:i\in \mathbb{N}\}$
is strongly $\mathrm{Coz}_{\delta}$-disjoint, there is a
completely $\mathrm{Coz}_{\delta}$-additive system $\{W_x: x\in
\bigcup_{k\in \mathbb{N}} \Omega_{2k}\}$ of
$\mathrm{Coz}_{\delta}$-sets of $X$ such that $x\in W_x$ for all
$x\in \bigcup_{k\in \mathbb{N}} \Omega_{2k}$, and moreover, the
sets  $W(i):=\bigcup \{W_x: x\in \Omega_{2i}\}$ $(i\in \mathbb{N})$ are $\mathrm{Coz}_{\delta}$-sets in $X$.

\medskip

{\it Claim 3. The completely $\mathrm{Coz}_{\delta}$-additive
system $\{W_x: x\in \bigcup_{k\in \mathbb{N}} \Omega_{2k}\}$ can
be chosen such that the following conditions are satisfied:

(i) all sets $W_x$ are zero-sets in $X$;

(ii) if $i\in \mathbb{N}$ and $x\in \Omega_{2i}$, then $W_x\cap
\bigcup\{\Omega_j:j<2i\}=\emptyset$.

Consequently, for any countable set $Z\subseteq \bigcup_{k\in
\mathbb{N}} \Omega_{2k}$, the set $\bigcup_{z\in Z} W_z$ is
$\mathrm{Zer}_{\sigma}$ and $\mathrm{Coz}_{\delta}$ in $X$ (in
particular, all sets $W(i)$ are $\mathrm{Zer}_{\sigma}$ and
$\mathrm{Coz}_{\delta}$ in $X$).}

\medskip

{\it Proof of Claim 3.} Since $X$ is a Tychonoff space and the
sequence $\{\Omega_i: i\in\mathbb{N}\}$ is disjoint, for every
$i\in \mathbb{N}$ there is a continuous function $f_i:
X\rightarrow [0,1]$ such that $f_i(\bigcup \{\Omega_j:
j<2i\})=\{0\}$ and $f_i(\Omega_{2i})=\{1\}$. Then
$f^{-1}_i([0,\frac{1}{2}))$ is a
$\mathrm{Zer}_{\sigma}$-set in $X$ and $\bigcup \{\Omega_j:
j<2i\}\subseteq f^{-1}_i([0,\frac{1}{2}))$. Therefore
the set $W'(i):=W(i)\setminus
f^{-1}_i([0,\frac{1}{2}))$ is a
$\mathrm{Coz}_{\delta}$-set in $X$ such that $\Omega_{2i}\subseteq
W'(i)$ and $\bigcup \{\Omega_j: j<2i\}\cap W'(i)=\emptyset$. By
Lemma \ref{lem101}, the sequence $\{W'(i):i\in \mathbb{N}\}$ forms
a completely $\mathrm{Coz}_{\delta}$-additive system. Replacing
$W(i)$ by $W'(i)$ if needed, we can assume that the completely
$\mathrm{Coz}_{\delta}$-additive system $\{W(i): i\in
\mathbb{N}\}$ satisfies the following condition

(2) \, \, $W(i)\cap \bigcup \{\Omega_j: j<2i\}=\emptyset$ for
every $i\in \mathbb{N}$.

Now we fix an arbitrary $x\in \bigcup_{k\in \mathbb{N}}
\Omega_{2k}$. Then $W_x$ is a $\mathrm{Coz}_{\delta}$-set in $X$.
Therefore $X\setminus W_x=\bigcup_{i\in \mathbb{N}} Z_i\in
\mathrm{Zer}_{\sigma}(X)$, where all $Z_i$ are zero-sets in $X$.
Since $X$ is Tychonoff and $x\not\in Z_i$, by Theorem 1.5.13 of \protect\cite{Eng}, for every $i\in \mathbb{N}$ there is a continuous
function $q_i: X\rightarrow [0,1]$ such that $q_i(x)=0$ and
$q_i(Z_i)=\{1\}$. Then the set $B_x:=\bigcap_{i\in \mathbb{N}}
q_i^{-1}(0)$ is a zero-set in $X$ such that $x\in B_x\subseteq
W_x$. By Lemma \ref{lem101}, the family $\{B_x: x\in \bigcup_{k\in
\mathbb{N}} \Omega_{2k}\}$ forms a disjoint completely
$\mathrm{Coz}_{\delta}$-additive system of zero-sets in $X$.

Finally, the claim follows if to replace $W_x$ by $B_x$ if needed.
\, \, \, \, \, \, \, \,  $\Box$

\medskip

Let $S=\bigcup_{i\in \mathbb{N}} W(i)$. Then, by Claim
3, $F:=X\setminus S\in \mathrm{Coz}_{\delta}(X)$.

\medskip

{\it Claim 4. There are a sequence $\{f_i: i\in
\mathbb{N}\}\subset B_1(X)$, a strictly increasing sequence
$\{r_i: i\in \mathbb{N}\}\subseteq \mathbb{N}$ with $r_1=1$, a
sequence $\{b_i: i\in \mathbb{N}\}$, and a sequence $\{U_i: i\in
\mathbb{N}\}$ of basic open sets in $\mathbb{R}^X$ such that

(e) $U_i\in \mu_{2r_i}$ for every $i\in \mathbb{N}$;

(f) $\|f_i\|_X\leq l_{2r_i}$ for every $i\in \mathbb{N}$;

(g) $\|f_{i+1}-f_i\|_{C_i}\leq \frac{1}{2^{b_i}}$ for every $i\in
\mathbb{N}$ where $C_i:=F \cup \bigcup_{j=1}^{r_i} W(j)$;

(h) $f_j\in U_i$ for every $j\geq i\geq 1$;

(k) $b_{i+1}-b_{i}\rightarrow \infty$.}

\medskip

{\it Proof of Claim 4.} For $i=1$, take an arbitrary $U_1\in
\mu_2=\mu_{2r_1}$. Since $X$ is Tychonoff and $supp(U_1)$ is
finite (because $U_1$ is a basic open set), (c) implies that there
is a continuous function $\widetilde{f}_1: X\rightarrow
\mathbb{R}$ such that $\widetilde{f}_1\in U_1$ and
$\|\widetilde{f}_1\|_{X}<l_{2r_1}$; in particular,
$a_1:=\|\widetilde{f}_1\|_{supp(U_1)}<l_{2r_1}$.

Since $\{W_z:
z\in \bigcup_{k\in \mathbb{N}} \Omega_{2k}\}$ is pairwise disjoint
(see Claim 3), we can define a function $f_1:X\rightarrow
\mathbb{R}$ by

$$ f_1(x):=\left\{
\begin{array}{lcl}
\widetilde{f}_{1}(x), \, \, \, \, \, if \, \, \, \, x\in F;\\
\widetilde{f}_{1}(z), \, \, \, if \, \, \, x\in W_z \, \,  for \, some \, \, z\in \bigcup_{k\in \mathbb{N}} \Omega_{2k}. \\
\end{array}
\right.
$$

Since $\{W_z: z\in \bigcup_{k\in \mathbb{N}} \Omega_{2k}\}\cup \{F\}$ is a disjoint completely $Coz_{\delta}$-additive
system, $f_1$ is constant on
each $W_z$ and $f_1\upharpoonright F$ is continuous, $f_1$ is a
Baire-one function on $X$. Let us explain in more detail. Let $O$ be an open set in $\mathbb{R}$. Then

$f_1^{-1}(O)=\biggl(\widetilde{f}_1^{-1}(O)\cap F\biggl)\cup \biggl(\bigcup \{W_z: f_{1}^{-1}(O)\cap W_z\neq \emptyset \}\biggl)$.

Since $F$ and $\widetilde{f}_1^{-1}(O)$ are $Zer_{\sigma}$-sets, the set $\widetilde{f}_1^{-1}(O)\cap F\in Zer_{\sigma}$.
Since $\{W_z: z\in \bigcup_{k\in \mathbb{N}} \Omega_{2k}\}\cup \{F\}$ is a disjoint completely $Coz_{\delta}$-additive
system, the set $\bigcup \{W_z: f_{1}^{-1}(O)\cap W_z\neq \emptyset \}\in Zer_{\sigma}$. Thus, we get that $f_1^{-1}(O)\in Zer_{\sigma}$ and hence $f_1\in B_1(X)$.

\medskip

The inequality
$\|\widetilde{f}_1\|_{X}<l_{2r_1}$ implies that
$\|f_1\|_{X}<l_{2r_1}$. To check that $f_1\in U_1$, we recall
that, by (ii) of Claim 3, $\Omega_1\cap W_x=\emptyset$ for every
$x\in \bigcup_{k\in \mathbb{N}} \Omega_{2k}$, and hence
$\Omega_1\subseteq F$. Therefore $f_1(x)=\widetilde{f}_1(x)$ for
every $x\in \Omega_1\cup \Omega_2$. Since $U_1\in \mu_2$, the
condition (b) implies $supp(U_1)\subseteq \Omega_1\cup \Omega_2$.
Thus $f_1\in U_1$.  Recall that $a_1<l_{2r_1}$. Now it is
clear that for a sufficiently large $b_1$ (such that
$\frac{1}{2^{b_1}}<l_{2r_1}-a_1$), we obtain that the condition
$\|f-f_1\|_{F\cup \Omega_2}\leq \frac{1}{2^{b_1}}$ implies $f\in U_1$.

\medskip
%--------------------------------------------------------------------------------------

Let $i=2$. Let us construct
$r_{2}$, $U_{2}$, $f_{2}$ and $b_{2}$, such that (e)-(k)
are satisfied as well. Choose $r_{2}\in\mathbb{N}$ such that
$r_{2}>r_1$ and $\frac{1}{2r_{2}-1}<\frac{1}{2^{b_1}}$.

Define $\varphi:=f_1\upharpoonright D$ where
$D:=\bigcup_{i=1}^{2r_{2}-1}\Omega_i$. Then $\varphi$ is a
function from $D$ to $\mathbb{R}$ such that $\|\varphi\|_D\leq
l_{2r_1}\leq l_{2r_{2}-1}$, and hence, by (d), there exists a
continuous function $\widetilde{f}_{2}\in \bigcup
\mu_{2r_{2}}$ such that  $\|\widetilde{f}_{2}\|\leq
l_{2r_{2}-1}$, $|\varphi(x)-\widetilde{f}_{2}(x)|<\frac{1}{2r_{2}-1}<\frac{1}{2^{b_1}}$
for each $x\in D$. Since $\widetilde{f}_{1}$ is continuous, we can assume that $|\widetilde{f}_{2}(x)-f_1(x)|=|\widetilde{f}_{2}(x)-\widetilde{f}_1(x)|<\frac{1}{2^{b_1}}$ for every $x\in F$.

Let $U_{2}\in \mu_{2r_{2}}$ be such that
$\widetilde{f}_{2}\in U_{2}$. Since $\{W_z: z\in \bigcup_{k\in
\mathbb{N}} \Omega_{2k}\}$ is pairwise disjoint (see Claim 3), one
can define a function $f_{2}:X\rightarrow \mathbb{R}$ by

$$ f_{2}(x):=\left\{
\begin{array}{lcl}
\widetilde{f}_{2}(x), \, \, \, \, \, if \, \, \, \, x\in F;\\
\widetilde{f}_{2}(z), \, \, \, \, \, if \, \, \, \, x\in W_z \, \,  for \, some \, \, z\in \bigcup_{k\in \mathbb{N}} \Omega_{2k}. \\
\end{array}
\right.
$$

Since $\{W_z: z\in \bigcup_{k\in \mathbb{N}} \Omega_{2k}\}\cup \{F\}$ is a disjoint completely $Coz_{\delta}$-additive
system, $f_2$ is constant on
each $W_z$ and $f_2\upharpoonright F$ is continuous, $f_2$ is a
Baire-one function on $X$.

  Since
$\|\widetilde{f}_{2}\|_X\leq l_{2r_{2}-1}$, it follows that
$\|f_{2}\|_{X}<l_{2r_{2}}$. Now we check that $f_{2}\in
U_{2}$. By (ii) of Claim 3, we have
$(\bigcup_{i=1}^{2r_{2}-1}\Omega_i)\cap W_x=\emptyset$ for every
$x\in \bigcup_{i\geq r_{2}}\Omega_{2i}$, and hence $\bigcup\limits_{i=1}^{2r_{2}-1}\Omega_i\subseteq F\cup \bigcup\limits_{i=1}^{r_{2}-1} \Omega_{2i}$.

Therefore $f_{2}(x)=\widetilde{f}_{2}(x)$ for every $x\in
\bigcup_{i=1}^{2r_{2}} \Omega_i$. Since $U_{2}\in
\mu_{2r_{2}}$, the condition (b) implies $supp(U_{2})\subseteq
\bigcup_{i=1}^{2r_{2}} \Omega_i$.
Thus $f_{2}\in U_{2}$.

Now we choose $b_{2}>b_1+1$ (that implies the condition (k) of
the claim) as follows. By the previous paragraph we have

$supp(U_{2})\subseteq
\bigcup\limits_{i=1}^{2r_{2}}\Omega_i\subseteq F\cup \bigcup
\{W_z: z\in \bigcup\limits_{i=1}^{r_{2}-1}\Omega_{2i}\}\cup
\bigcup\{W_x: x\in\Omega_{r_{2}}\}$.

Taking into account that $U_{2}$ is a basic open set,
$f_{2}\in U_{2}$ and $\|f_{2}\|_{X}<l_{2r_{2}}$, it is
clear that for a sufficiently small $\varepsilon>0$, it follows
that if $\|g-f_{2}\|_{C_2}<\varepsilon$, then $g\in U_{2}$.
Take $b_{2}>b_1+1$ such that
$\frac{1}{2^{b_{2}}}<\varepsilon$. Finally, since
$r_1<r_{2}-1$ it follows that $C_1\subseteq C_{2}$ and hence $\|f_{2}-f_1\|_{C_1}\leq \frac{1}{2^{b_1}}$.

\medskip

%--------------------------------------------------------------------------------------------------------------

To check the step of induction, assume that we have found
$f_1,...,f_k\in B_1(X)$,  $1=r_1<...<r_k$ in $\mathbb{N}$,
$b_1<...<b_k$, and basic open sets $U_1$,...,$U_k$ in
$\mathbb{R}^X$ such that (e)-(k) are satisfied. Let us construct
$r_{k+1}$, $U_{k+1}$, $f_{k+1}$ and $b_{k+1}$, such that (e)-(k)
are satisfied as well. Choose $r_{k+1}\in\mathbb{N}$ such that
$r_{k+1}>r_k$ and $\frac{1}{2r_{k+1}-1}<\frac{1}{2^{b_k}}$.

Define $\varphi:=f_k\upharpoonright D$ where
$D:=\bigcup_{i=1}^{2r_{k+1}-1}\Omega_i$. Then $\varphi$ is a
function from $D$ to $\mathbb{R}$ such that $\|\varphi\|_D\leq
l_{2r_k}\leq l_{2r_{k+1}-1}$, and hence, by (d), there exists a
continuous function $\widetilde{f}_{k+1}\in \bigcup
\mu_{2r_{k+1}}$ such that  $\|\widetilde{f}_{k+1}\|\leq
l_{2r_{k+1}-1}$, $|\varphi(x)-\widetilde{f}_{k+1}(x)|<\frac{1}{2r_{k+1}-1}<\frac{1}{2^{b_k}}$
for each $x\in D$. Since $\widetilde{f}_{k}$ is continuous, we can assume that $|\widetilde{f}_{k+1}(x)-f_k(x)|=|\widetilde{f}_{k+1}(x)-\widetilde{f}_k(x)|<\frac{1}{2^{b_k}}$ for every $x\in F$.

Let $U_{k+1}\in \mu_{2r_{k+1}}$ be such that
$\widetilde{f}_{k+1}\in U_{k+1}$. Since $\{W_z: z\in \bigcup_{k\in
\mathbb{N}} \Omega_{2k}\}$ is pairwise disjoint (see Claim 3), one
can define a function $f_{k+1}:X\rightarrow \mathbb{R}$ by

$$ f_{k+1}(x):=\left\{
\begin{array}{lcl}
\widetilde{f}_{k+1}(x), \, \, \, \, \, if \, \, \, \, x\in F;\\
\widetilde{f}_{k+1}(z), \, \, \, \, \, if \, \, \, \, x\in W_z \, \,  for \, some \, \, z\in \bigcup_{k\in \mathbb{N}} \Omega_{2k}. \\
\end{array}
\right.
$$

Since $\{W_z: z\in \bigcup_{k\in \mathbb{N}} \Omega_{2k}\}\cup \{F\}$ is a disjoint completely $Coz_{\delta}$-additive
system, $f_{k+1}$ is constant on
each $W_z$ and $f_{k+1}\upharpoonright F$ is continuous, $f_{k+1}$ is a
Baire-one function on $X$. Let us explain in more detail. Let $O$ be an open set in $\mathbb{R}$. Then

$f_{k+1}^{-1}(O)=\biggl(\widetilde{f}_{k+1}^{-1}(O)\cap F\biggl)\cup \biggl(\bigcup \{W_z: f_{k+1}^{-1}(O)\cap W_z\neq \emptyset \}\biggl)$.

Since $F$ and $\widetilde{f}_{k+1}^{-1}(O)$ are $Zer_{\sigma}$-sets, the set $\widetilde{f}_{k+1}^{-1}(O)\cap F\in Zer_{\sigma}$.
Since $\{W_z: z\in \bigcup_{k\in \mathbb{N}} \Omega_{2k}\}\cup \{F\}$ is a disjoint completely $Coz_{\delta}$-additive
system, the set $\bigcup \{W_z: f_{k+1}^{-1}(O)\cap W_z\neq \emptyset \}\in Zer_{\sigma}$. Thus, we get that $f_{k+1}^{-1}(O)\in Zer_{\sigma}$ and hence $f_{k+1}\in B_1(X)$.

  Since
$\|\widetilde{f}_{k+1}\|\leq l_{2r_{k+1}-1}$, it follows that
$\|f_{k+1}\|_{X}<l_{2r_{k+1}}$. Now we check that $f_{k+1}\in
U_{k+1}$. By (ii) of Claim 3, we have
$(\bigcup_{i=1}^{2r_{k+1}-1}\Omega_i)\cap W_x=\emptyset$ for every
$x\in \bigcup_{i\geq r_{k+1}}\Omega_{2i}$, and hence $\bigcup\limits_{i=1}^{2r_{k+1}-1}\Omega_i\subseteq F\cup \bigcup\limits_{i=1}^{r_{k+1}-1} \Omega_{2i}$.

Therefore $f_{k+1}(x)=\widetilde{f}_{k+1}(x)$ for every $x\in
\bigcup_{i=1}^{2r_{k+1}} \Omega_i$. Since $U_{k+1}\in
\mu_{2r_{k+1}}$, the condition (b) implies $supp(U_{k+1})\subseteq
\bigcup_{i=1}^{2r_{k+1}} \Omega_i$. Thus $f_{k+1}\in U_{k+1}$.

Now we choose $b_{k+1}>b_k+k$ (that implies the condition (k) of
the claim) as follows. By the previous paragraph we have

$supp(U_{k+1})\subseteq
\bigcup\limits_{i=1}^{2r_{k+1}}\Omega_i\subseteq F\cup \bigcup
\{W_z: z\in \bigcup\limits_{i=1}^{r_{k+1}-1}\Omega_{2i}\}\cup
\bigcup\{W_x: x\in\Omega_{r_{k+1}}\}$.

Taking into account that $U_{k+1}$ is a basic open set,
$f_{k+1}\in U_{k+1}$ and $\|f_{k+1}\|_{X}<l_{2r_{k+1}}$, it is
clear that for a sufficiently small $\varepsilon>0$, it follows
that if $\|g-f_{k+1}\|_{C_k}<\varepsilon$, then $g\in U_{k+1}$.
Take $b_{k+1}>b_k+k$ such that
$\frac{1}{2^{b_{k+1}}}<\varepsilon$. Finally, since
$r_k<r_{k+1}-1$ it follows that $C_k=F\cup\bigcup_{j=1}^{r_k} W(j)\subseteq C_{k+1}$

and hence $\|f_{k+1}-f_k\|_{C_k}\leq \frac{1}{2^{b_k}}$. \, \, \,
\, \, \, \, \, \, \, \, \, \, \, \, \, \, \, \, \, \, \, \, \, \,
\, \, \, $\Box$

\medskip

It follows from Claim 4(g) that the sequence $\{f_i: i\in
\mathbb{N}\}\subseteq B_1(X)$ converges pointwise to some function
$f: X\rightarrow \mathbb{R}$.

%-------------------------------------

To show that the function $f$ is Baire-one, let $V=(a,b)$ be an
open interval of $\mathbb{R}$.

Since all functions $\widetilde{f}_{i}$ are continuous, $\{W_z: z\in \bigcup_{k\in \mathbb{N}} \Omega_{2k}\}\cup\{F\}$ is a disjoint completely $Coz_{\delta}$-additive
system and $f_{i}$ is constant on
each $W_z$ and $C_k\in
\mathrm{Zer}_{\sigma}(X)$, the set

$f^{-1}(V)\cap C_k=\Bigl(\bigcup\limits_{n>\frac{2}{b-a}} \,
\bigcap\limits_{l>n+k}
f^{-1}_{l}([a+\frac{1}{n},b-\frac{1}{n}])\Bigr)\cap C_k=$

$=\Bigl( \Bigl(\bigcup\limits_{n>\frac{2}{b-a}} \, \bigcap\limits_{l>n+k}
\widetilde{f}^{-1}_{l}([a+\frac{1}{n},b-\frac{1}{n}])\Bigr)\cap F\Bigr)\bigcup \\ \bigcup \Bigl( \Bigl(\bigcup\limits_{n>\frac{2}{b-a}} \, \bigcap\limits_{l>n+k}
f^{-1}_{l}([a+\frac{1}{n},b-\frac{1}{n}])\Bigr)\cap\bigcup_{j=1}^{r_k} W(j)\Bigr)\in \mathrm{Zer}_{\sigma}(X)$.

 It follows that
$f^{-1}(V)=\bigcup\limits_{k\in \mathbb{N}} (f^{-1}(V)\cap
C_k)\in\mathrm{Zer}_{\sigma}(X)$ and, hence, $f\in B_1(X)$.

By Claim 4(h),  $f\in \overline{U_i}$ for each $i\in \mathbb{N}$.
But, then, by Claim 2(b), $f\not\in \bigcup\limits_{i=1}^{\infty}
F_i=B_1(X)$, it is a contradiction.

$(2)\Rightarrow(1)$. Let $\{\Delta_n: n\in \mathbb{N}\}$ satisfy
(2). For every $m\in \mathbb{N}$, define

\, \, $F_m:=\Bigl\{f\in B_1(X): \sup\limits_{n\in
\mathbb{N}}\Bigl(\min\{|f(x)|: x\in\Delta_n\}\Bigr)\leq m\Bigr\}$.

By (2) we have $B_1(X)=\bigcup_{m\in \mathbb{N}} F_m$. Therefore
it remains to prove that $F_m$ is nowhere dense in $B_1(X)$. First
we show that $F_m$ is closed in $B_1(X)$. Indeed, let $f\in
B_1(X)$ be such that $|f(z)|>m$ for some $n\in \mathbb{N}$ and
each $z\in \Delta_n$. Set $\varepsilon:=\frac{1}{2}(\min\{|f(x)|:
x\in \Delta_n\}-m)$. Then the standard neighborhood
$f+[\Delta_n;\varepsilon]:=\{g\in B_1(X): |f(x)-g(x)|<\varepsilon$
for every $x\in \Delta_n\}$ of $f$ does not intersect $F_m$. Thus
$F_m$ is closed.

To show that the closed set $F_m$ is nowhere dense in $B_1(X)$,
suppose for a contradiction that $F_m$ contains a standard
neighborhood $g+[A;\delta]$ of some $g\in F_m$, where $A\subset X$
is finite and $\delta>0$. Since the sequence $\{\Delta_n:
n\in\mathbb{N}\}$ is disjoint, there is $n_0\in \mathbb{N}$ such
that $\Delta_{n_0}\cap A=\emptyset$. As $X$ is Tychonoff, there is
$h\in C(X)$ such that $h\upharpoonright_A=g$ and
$h\upharpoonright_{\Delta_{n_0}}=2m$. It is clear that $h\in
g+[A;\delta]\subseteq F_m$, however

$\sup\limits_{n\in \mathbb{N}}\Bigl(\min\{|h(x)|: x\in
\Delta_n\}\Bigr)\geq \min\{|h(x)|: x\in \Delta_{n_0}\}=2m>m$

and hence $h\not\in F_m$, a contradiction. Thus $F_m$ is nowhere
dense in $B_1(X)$ and hence $B_1(X)$ is meager.

$(2)\Rightarrow(3)$.  Let $\Delta=\{\Delta_n : n\in\mathbb{N}\}$
be a sequence satisfying (2). Assume that $\{\Delta_{n_i}: i\in
\mathbb{N} \}$ is a strongly $Coz_{\delta}$-disjoint subsequence
of $\Delta$. Then, there is a pairwise disjoint collection $\{F_a:
a\in \bigcup \Delta_{n_i}\}$ of $Coz_{\delta}$ neighborhoods of
$a\in \bigcup \Delta_{n_i}$ such that $\bigcup\limits_{b\in B}
F_b\in Coz_{\delta}$ for each $B\subseteq \bigcup \Delta_{n_i}$.

Let $F=\bigcup \{F_a : a\in \bigcup \Delta_{n_i}\}$ and $S=\{a_j:
j\in\mathbb{N}\}=\bigcup \Delta_{n_i}$. Then, the function $f: F
\rightarrow \mathbb{R}$ such that $f(F_{a_i})=i$ for every $i\in
\mathbb{N}$  is constant on each $F_a$. By Lemma \ref{lem32},
there is a Baire-one function $g$ on $X$ extending $f$. Note that
$\sup\limits_{n}\Bigl(\min \{|g(x)|: x\in
\Delta_n\}\Bigr)=\infty$, a contradiction.

$(3)\Rightarrow(2)$. Suppose for a contradiction that
$\Delta=\{\Delta_n : n\in\mathbb{N}\}$ does not contain strongly
$Coz_{\delta}$-disjoint subsequence and there is $f\in B_1(X)$
such that $\sup\limits_{n\in \mathbb{N}}\Bigl(\min \{|f(x)|: x\in
\Delta_n\}\Bigr)=\infty$.

Consider $\{\Delta_{n_k}: k\in\mathbb{N}\}$ such that $m_k<\min
\{|f(x)|: x\in \Delta_{n_k}\}$ and $\max \{|f(x)|: x\in
\Delta_{n_k}\}<m_{k+1}$ for every $k\in \mathbb{N}$ where
$m_k\rightarrow \infty$ $(k\rightarrow \infty)$. Then
$(f(\Delta_{n_k}))_k$ is a discrete family of finite (and hence
closed) subset of $\mathbb{R}$. It follows that this family is
completely closed-additive (the union of any subfamily is closed).
Since $f$ is Baire-one, we deduce that the family
$(f^{-1}(f(\Delta_{n_k})))_k$ is completely
$\mathrm{Coz}_{\delta}$-additive. By Proposition \ref{pr34}, this
immediately shows that the sequence $(\Delta_{n_k})_k$ is strongly
$Coz_{\delta}$-disjoint, a contradiction.

$(4)\Rightarrow (3)$.  Let $\sigma$ be a winning strategy for ONE
player in the game $Game(X)$. A sequence $(B_0,...,B_n)$ of
disjoint finite subsets of $X$ is called {\it correct} if
$B_{2k+1}=\sigma(B_0,...,B_{2k})$, where $2k+1\leq n$. By
induction, we construct a disjoint sequence $\{\Delta_n : n\in
\mathbb{N}\}$ of finite subsets of $X$.

Let $\Delta_1=B_0$ be an arbitrary finite subset of $X$.
$\Delta_2=B_1=\sigma(B_0)$. Assume that $\Delta_0$, ...,
$\Delta_n$ are constructed then

$\Delta_{n+1}:=\bigcup\{\sigma(B_0,...,B_{2m})$:
$(B_0,...,B_{2m})$ is correct  and
$\bigcup\limits_{j=0}^{2m}B_j=\bigcup\limits_{i=1}^n \Delta_i$,
$m\leq n\}$.

  We claim that for any subsequence $\{\Delta_{n_k}: k\in \mathbb{N}\}$,
$1\leq n_1<n_2<...$ there is a sequence $\{D_k: k\in\mathbb{N}\}$
satisfying the conditions:

(a) $(D_0,...,D_k)$ is a correct sequence for each $k$;

(b) $D_{2p+1}\subseteq \Delta_{n_{p+1}}$ for any $p\geq 0$ and
$\bigcup\limits_{i=0}^{2m} D_i=\bigcup\limits_{i=1}^{n_{m+1}-1}
\Delta_i$ for any $m\geq 1$.

By induction, we construct a sequence $\{D_k: k\in \mathbb{N}\}$.
Let $D_0=\bigcup\limits_{i=1}^{n_1-1} \Delta_i$. Assume that
$D_0$, ..., $D_{2m}$ are constructed. Let's construct the sets
$D_{2m+1}$, $D_{2m+2}$. By condition (a), $(D_0,...,D_{2m})$ is
correct, and, by condition (b), $\bigcup\limits_{i=0}^{2m}
D_i=\bigcup\limits_{i=1}^{n_{m+1}-1} \Delta_i$. Then, by
constructing the sequence $\{\Delta_n: n\in\mathbb{N}\}$,
$\sigma(D_0,...,D_{2m})\subseteq \Delta_{n_{m+1}}$. Let
$D_{2m+1}=\sigma(D_0,...,D_{2m})$,
$D_{2m+2}=\bigcup\limits_{i=n_m+1}^{n_{m+2}-1} \Delta_i\setminus
D_{2m+1}$. It is not difficult to verify that the sequence
$(D_0,...,D_{2m+2})$ also satisfies conditions (a) and (b).

Since $\sigma$ is a winning strategy for ONE player in the game
$Game(X)$, the sequence $\{D_{2p+1}: p\in \mathbb{N}\}$ is not
strongly $Coz_{\delta}$-disjoint, hence, by condition (b),
$\{\Delta_{n_k}: k\in \mathbb{N}\}$ also is not strongly
$Coz_{\delta}$-disjoint. Since the subsequence $\{\Delta_{n_k}:
k\in \mathbb{N}\}$ was arbitrary, it follows that the sequence
$\{\Delta_{n}: n\in \mathbb{N}\}$ satisfies the condition (3).

$(3)\Rightarrow (4)$. Let $\{\Delta_{n}: n\in \mathbb{N}\}$ be a
sequence satisfying (3). For every finite set $P\subseteq X$, put
$n(P)=\min\{m: \Delta_{m}\cap P=\emptyset \}$. Then the strategy
of player $ONE$: $S_{2n+1}=\sigma(S_0,...,S_{2n})=\Delta_{n(L)}$
where $L=\bigcup\limits_{i=0}^{2n} S_i$, is winning.

%$(1)\Leftrightarrow(5)$. By Oxtoby's Theorem \ref{thox}.

\end{proof}

Since a non-meager space $B_1(X)$ is Baire, we have the following
result.

\medskip

\begin{theorem}\label{cor22} Let $X$ be a topological space. The following assertions are
equivalent:

1. $B_1(X)$ is Baire;

2. every pairwise disjoint sequence of non-empty finite subsets of
$X$ has a strongly $Coz_{\delta}$-disjoint subsequence;

3. ONE has no winning strategy in the game $G_{I}(B_1(X))$;

4. ONE has no winning strategy in the game $Game(X)$.
\end{theorem}

\medskip

 Note that if $Y\subset X$ and $A\in Coz_{\delta}(X)$ then $A\cap Y\in
Coz_{\delta}(Y)$.

\medskip

\begin{corollary}\label{cor55}{\it Let $B_1(X)$ be a Baire space and $Y\subseteq
X$. Then $B_1(Y)$ is Baire.}
\end{corollary}

\medskip

It is well-known that there are Baire spaces $X$ and $Y$ such that
$X\times Y$ is not Baire \protect\cite{Ox1}. For the product
$\prod\limits_{\alpha\in A} B_1(X_{\alpha})$ we have the following
result.

\medskip

\begin{corollary} If $B_1(X_{\alpha})$ is Baire  for all $\alpha\in
A$, then $\prod\limits_{\alpha\in A} B_1(X_{\alpha})$ is Baire.
\end{corollary}

\begin{proof} It is well-known that $\prod\limits_{\alpha\in A}
C_p(X_{\alpha})\cong C_p(\bigoplus\limits_{\alpha\in A} X_{\alpha})$
\protect\cite{arh}. We claim that $\prod\limits_{\alpha\in A}
B_1(X_{\alpha})\cong B_1(\bigoplus\limits_{\alpha\in A} X_{\alpha})$.
To this end, we define a mapping $f\rightarrow f^*$ where $f\in \prod\limits_{\alpha\in A}
B_1(X_{\alpha})$ and $f^*\in B_1(\bigoplus\limits_{\alpha\in A} X_{\alpha})$.

Let $f\in \prod\limits_{\alpha\in A}
B_1(X_{\alpha})$. Then $f=(f_{\alpha})$ where $f_{\alpha}\in B_1(X_{\alpha})$ for each $\alpha\in A$.

Define $f^*: \bigoplus\limits_{\alpha\in A} X_{\alpha} \rightarrow \mathbb{R}$ by letting $f^*\upharpoonright X_{\alpha}=f_{\alpha}$ for each $\alpha\in A$.

We claim that  $f^*$ is a Baire-one function. For each $\alpha\in A$, there is a sequence $\{f_{n,\alpha}: X_{\alpha}\rightarrow \mathbb{R}\}$ of continuous functions such that $f_{n,\alpha}\rightarrow f_{\alpha}$ pointwise. For each $n$, define $f_n: \bigoplus\limits_{\alpha\in A} X_{\alpha} \rightarrow \mathbb{R}$ such that $f_n\upharpoonright X_{\alpha}=f_{n,\alpha}$ for each $\alpha\in A$. It follows that $f_n\rightarrow f^*$ pointwise. Thus, $f^*\in B_1(\bigoplus\limits_{\alpha\in A} X_{\alpha})$.

\medskip
The mapping $f\rightarrow f^*$ is one-to-one mapping from  $\prod\limits_{\alpha\in A} B_1(X_{\alpha})$ onto $B_1(\bigoplus\limits_{\alpha\in A} X_{\alpha})$.
Furthermore, both the mapping and its inverse are continuous. Thus, the mapping is a homeomorphism.

\medskip

Let $B_1(X_{\alpha})$ be a Baire space  for all $\alpha\in A$.
We claim that $B_1(\bigoplus\limits_{\alpha\in A} X_{\alpha})$ is
Baire.

It is clear that $B_1(\bigoplus\limits_{\alpha\in A} X_{\alpha})$
is Baire for a finite set $A$. Suppose that $|A|\geq \aleph_0$.

Let $\gamma=\{\Delta_n: n\in\mathbb{N}\}\subseteq
\bigoplus\limits_{\alpha\in A} X_{\alpha}$ be a pairwise disjoint
sequence of non-empty finite subsets of
$\bigoplus\limits_{\alpha\in A} X_{\alpha}$. Let $A_1$ be a finite
subset of $A$ such that $\Delta_1\subseteq \bigoplus \{X_{\alpha}:
\alpha\in A_1\}=Y_1$. Consider $\gamma\cap Y_1=\{\Delta_n\cap Y_1:
n\in \mathbb{N}\}$. There is $\gamma_1\subseteq \gamma$ such that
$|\gamma_1|=\aleph_0$ and $\gamma_1\cap Y_1$ is strongly
$Coz_{\delta}$-disjoint in $Y_1$ and hence in
$\bigoplus\limits_{\alpha\in A} X_{\alpha}$. Suppose that
$\Delta_1\in \gamma_1$. Let $\Delta_{n_2}\in \gamma_1\setminus
\Delta_1$. Choose $A_2\subseteq A$, $|A_2|<\aleph_0$ such that
$\Delta_{n_2}\subseteq Y_1\cup Y_2$, where $Y_2=\bigoplus
\{X_{\alpha}: \alpha\in A_2\}$. Consider $\gamma_1\cap Y_2$. There
is $\gamma_2\subseteq \gamma_1$, $|\gamma_2|=\aleph_0$ and
$\gamma_2\cap Y_2$ is strongly $Coz_{\delta}$-disjoint in $Y_2$.
Choose $\Delta_{n_3}\in \gamma_2\setminus
\{\Delta_1,\Delta_{n_2}\}$ and so on. We get the strongly
$Coz_{\delta}$-disjoint sequence $\{\Delta_{n_i}: i\in
\mathbb{N}\}$ in $\bigoplus\limits_{\alpha\in A} X_{\alpha}$ which
satisfies our requirements.
\end{proof}

\begin{proposition} Let $X$ be a Hausdorff space and let $\{F_i: i\in \mathbb{N}\}$  be a pairwise disjoint family of $Coz_{\delta}$ subsets of $X$ such
that $\bigcup F_i$ is a $Coz_{\delta}$-subset of $X$.
Then the following assertions are equivalent:

1. $\{F_i: i\in \mathbb{N}\}$ forms a completely
$Coz_{\delta}$-additive system;

2. there is a pairwise disjoint family $\{D_i: i\in\mathbb{N}\}$
of $Zer_{\sigma}$ subsets of $X$ such that $F_i\subseteq D_i$ for
each $i\in \mathbb{N}$.

\end{proposition}

\begin{proof} $(1)\Rightarrow(2)$. By Lemma \ref{lem32}, there is $f\in B_1(X)$ such that $f(F_i)=i$ for each
$i\in \mathbb{N}$. Let $D_i=f^{-1}(i-\frac{1}{2}, i+\frac{1}{2})$
for each $i\in \mathbb{N}$. Then $\{D_i: i\in\mathbb{N}\}$
satisfies our requirements.

$(2)\Rightarrow(1)$. Let $\{i_k: k\in \mathbb{N}\}\subseteq
\mathbb{N}$. Then $\bigcup \{F_{i_k}: k\in \mathbb{N}\}=\bigcup
F_i\setminus (\bigcup \{D_j: j\neq i_k, j,k\in\mathbb{N}\})$.
Since $\bigcup F_i\in Coz_{\delta}(X)$ and $\bigcup \{D_j: j\neq
i_k, j,k\in\mathbb{N}\}\in Zer_{\sigma}(X)$, the set $\bigcup
\{F_{i_k}: k\in \mathbb{N}\}\in Coz_{\delta}(X)$.
\end{proof}

Recall that the {\it pseudocharacter}  $\psi(X)$ of $X$ is the
smallest cardinal $\kappa$ such that every singleton
$\{x\}\subseteq X$ is the intersection of $\leq\kappa$ open sets
in $X$.

\medskip
\begin{lemma}\label{lem1} Let $X$ be a space of countable pseudocharacter and let
$A=\{a_i: i\in\mathbb{N}\}$ be a strongly $Coz_{\delta}$-disjoint
countable subset of $X$. Then $\{a_i: i\in\mathbb{N}\}$ forms a
completely $Coz_{\delta}$-additive system.
\end{lemma}

\begin{proof}   Since $X$ is a space
of countable pseudocharacter, $\{x\}\in Coz_{\delta}(X)$ for every
$x\in X$. Then, by Lemma \ref{lem101}, $\{a_i: i\in\mathbb{N}\}$
forms a completely $Coz_{\delta}$-additive system.
\end{proof}

Note that any Lindel\"{o}f $G_{\delta}$-set is $Coz_{\delta}$
(Proposition 4 (b) in \protect\cite{KS}). Then, by Lemma
\ref{lem1} and Theorem \ref{cor22}, we have the following result.

\medskip

\begin{theorem}\label{th20} Let $X$ be a space of countable pseudocharacter. A space $B_1(X)$ is Baire
if, and only if, every pairwise disjoint sequence $\{\Delta_i:
i\in \mathbb{N}\}$ of non-empty finite subsets of $X$ has a
subsequence $\{\Delta_{i_k}: k\in \mathbb{N}\}$ such that $\bigcup
\{\Delta_{i_k}: k\in\mathbb{N}\}$ is $G_{\delta}$.
\end{theorem}

\medskip

A set $A$ is called {\it locally non-$G_{\delta}$} if no nonempty
relatively open subset of $A$ is $G_{\delta}$.

\medskip
\begin{proposition}\label{th23} Let $B_1(X)$ be a Baire space for a metrizable space $X$. Then for any countable locally non-$G_{\delta}$ subset
$A\subseteq X$ there is a disjoint family $\{B_i: i\in
\mathbb{N}\}$ of countable subsets of $A$ such that $A=\bigcup
B_i$, $\overline{B_i}=\overline{A}$ and $B_i$ is a
$G_{\delta}$-set in $X$ for all $i\in \mathbb{N}$.
\end{proposition}

\begin{proof}
Since $\overline{A}$ is a separable metrizable space, by Theorem
4.3.5 from \protect\cite{Eng}, $\overline{A}$ is metrizable by a
totally bounded metric $\rho$. For each $\epsilon=\frac{1}{2^{n}}$
there exists a finite set $F_{n}$ of $\overline{A}$ which is
$\epsilon$-dense in $(\overline{A},\rho)$. It follows that for
each $n\in \mathbb{N}$ there exists a finite set $A_n\subset A$
which is $\frac{1}{2^{n-1}}$-dense in $(\overline{A},\rho)$. Since
$A$ is locally non-$G_{\delta}$, then we can assume that $A_i\cap
A_j=\emptyset$ for $i\neq j$, $i,j\in \mathbb{N}$.

 By Corollary \ref{cor22}, there exists a subsequence $\{A_{n_k}:
k\in\mathbb{N}\}\subseteq \{A_n: n\in \mathbb{N}\}$ such that
$\{A_{n_k}: k\in\mathbb{N}\}$ is strongly $Coz_{\delta}$-disjoint.
Let $B_1=\bigcup\{A_{n_k}: k\in \mathbb{N}\}$. Then
$\overline{B_1}=\overline{A}$ and, by Lemma \ref{lem1}, $B_1$ is a
$G_{\delta}$-set in $X$. Consider $A\setminus B_1$. Since $A$ is
not-$G_{\delta}$ then $A\setminus B_1=\{A_j: j\neq n_k, j,k\in
\mathbb{N}\}$ is not $G_{\delta}$, too. Since $A$ is locally
not-$G_{\delta}$ then $\overline{A\setminus B_1}=\overline{A}$.
Analogously, choose $B_2\subset (A\setminus B_1)$ such that
$B_2=\{A_{j_s}: s\in \mathbb{N}\}$ is a $G_{\delta}$-set in $X $
and $\overline{B_2}=\overline{A}$. By induction, choose
$B_j\subset(A\setminus \bigcup\{B_i: i=1,..,j-1\})$ such that
$B_j$ is a $G_{\delta}$-set in $X$ and
$\overline{B_j}=\overline{A}$. Note that the induction is an
infinite because $A$ is not $G_{\delta}$. We get the family
$\{B_i: i\in \mathbb{N}\}$ which satisfies our requirements. This
conclude the proof.
\end{proof}

\begin{theorem} If $X$ is metrizable and $B_1(X)$ is a Baire
space, then each separable subset of $X$ without isolated points
is meager (of first category) in itself.
\end{theorem}

\begin{proof} Let $Y$ be a separable subset of $X$ without isolated points. Assume that $Y$ is of second category in itself.
Then there is $Z\subseteq Y$ relatively open which is a Baire
space. In particular, $Z$ is locally uncountable (as it has no
isolated points). Let $A\subseteq Z$ be a countable dense subset.

We claim that  $A$ is locally non-$G_{\delta}$. Assume there is a nonempty relatively open $B\subseteq A$ which is
$G_{\delta}$. Let $U\subseteq X$ be open such that $U\cap A=B$.
Then $B$ is dense in $U\cap Z$. Moreover, $U\cap (Z\setminus B)$
is also dense in $U\cap Z$. Since $U\cap (Z\setminus B)$ is
$G_{\delta}$ in $U\cap Z$ (as $B$ is countable), we deduce that
the Baire space $U\cap Z$ has two disjoint dense $G_{\delta}$
sets, a contradiction. Thus, $A$ is locally non-$G_{\delta}$.

\medskip
By Proposition \ref{th23}, there is a disjoint family $\{B_i: i\in
\mathbb{N}\}$ of countable subsets of $A$ such that $A=\bigcup
B_i$, $\overline{B_i}=\overline{A}$ and $B_i$ is a
$G_{\delta}$-set in $X$ for all $i\in \mathbb{N}$. Then each $B_i$ is a
$G_{\delta}$-set in $Z$. Thus, we deduce that the Baire space $Z$ has disjoint dense $G_{\delta}$
sets $B_i$, a contradiction.

\end{proof}

\begin{corollary}{\it If $X$ is a separable metrizable space without isolated points such that
$B_1(X)$ is Baire, then $X$ is meager}.
\end{corollary}

In particular, a space $B_1(X)$ is meager for any uncountable
Polish space $X$ without isolated points.

\bigskip
Note that if $X=\mathbb{N}$ is a set of natural numbers ($X$ is Polish) then $B_1(X)=\mathbb{R}^{\omega}$ and, hence, $B_1(X)$ is Baire.

\bigskip

Let us recall that a cover $\mathcal{U}$ of a set $X$ is called

$\bullet$ an $\omega$-cover if each finite set $F\subseteq X$ is
contained in some $U\in \mathcal{U}$;

$\bullet$ a $\gamma$-cover if for any $x\in X$ the set $\{U\in
\mathcal{U}: x\not\in U\}$ is finite.

\medskip

A topological space $X$ is called a {\it $\gamma$-space} if each
countable open $\omega$-cover $\mathcal{U}$ of $X$ contains a
$\gamma$-subcover of $X$. $\gamma$-Spaces were introduced by
Gerlits and Nagy in \protect\cite{GN} and are important in the
theory of function spaces as they are exactly those $X$ for which
the space $C_p(X)$ has the Fr\'{e}chet-Urysohn property
\protect\cite{GN2}.

\medskip

%Recall that a topological space $X$ is {\it functionally
%Hausdorff} if for any distinct points $x,y\in X$ there exists a
%continuous function $f:X\rightarrow \mathbb{R}$ such that
%$f(x)\neq f(y)$.

\begin{theorem}\label{th55} Let $X$ be a $\gamma$-space. Then $B_1(X)$ is Baire.
\end{theorem}

\begin{proof}
We can assume that $X$ is a Tychonoff space otherwise we apply the
Tychonoff functor to the space $X$.

 By Theorem \ref{cor22}, it suffices to prove that every pairwise disjoint
sequence of non-empty finite subsets of $X$ has a strongly
$Coz_{\delta}$-disjoint subsequence.  Let $\{\Delta_n: n\in
\mathbb{N}\}$ be a pairwise disjoint sequence of non-empty finite
subsets of $X$. Since $X$ is a Tychonoff space, there exists a
pairwise disjoint family $\{F_a: a\in \bigcup \Delta_n\}$ of
zero-sets of $X$ such that $a\in F_a$ for each $a\in \bigcup
\Delta_n$. Indeed, let $a,b\in \bigcup \Delta_n$ and $a\neq b$.
Then, there is a continuous function $f_{a,b}:X\rightarrow
\mathbb{R}$ such that $f_{a,b}(a)\neq f_{a,b}(b)$. Note that
$\{f_{a,b}: a,b\in \bigcup \Delta_n \}$ is a countable family. Let
$F_a=\bigcap \{f^{-1}_{a,b}(f_{a,b}(a)): b\in \bigcup
\Delta_n\setminus \{a\}\}$. Note that $F_a$ is a zero-set of $X$
and $F_a\cap F_b=\emptyset$ for $a\neq b$.

Let $Q_n=\bigcup \{F_a: a\in \Delta_n\}$. Clearly, $Q_n$ is a
zero-set of $X$, $\Delta_n\subseteq Q_n$ for each $n\in
\mathbb{N}$ and $Q_n\cap Q_{n'}=\emptyset$ for $n\neq n'$.

For every $n\in \mathbb{N}$ the zero-set $Q_n$ can be represented
as $Q_n=\bigcap \{W_{n,i}: i\in\mathbb{N}\}=\bigcap \{S_{n,i}:
i\in \mathbb{N}\}$ where $W_{n,i+1}\subseteq S_{n,i+1}\subseteq
W_{n,i}$, $W_{n,i}$ is a cozero-set of $X$ and $S_{n,i+1}$ is a
zero-set of $X$ for each $i\in\mathbb{N}$.

  Then $\{X\setminus S_{n,k}: k,n\in \mathbb{N} \}$ is an open
$\omega$-cover of $X$. Since $X$ is a $\gamma$-space, there is a
$\gamma$-subcover $\mathcal{V}=\{X\setminus S_{n_i,k_i}: i\in
\mathbb{N}\}$. The $\gamma$-subcover $\mathcal{V}$ leads to the subsequence $\{\Delta_{n_i}: i\in \mathbb{N}\}$ and the subsequence $\{Q_{n_i}: i\in \mathbb{N}\}$ with $\Delta_{n_i}\subseteq Q_{n_i}$ for each $i$.

Since any infinite subset of $\mathcal{V}$ is a
$\gamma$-cover of $X$, shrinking $\mathcal{V}$, if necessary, we
may assume that $n_{i+1}>n_i$ for all $i\in \mathbb{N}$. Also, we
may assume that $k_{i+1}>k_i$ for all $i$ by enlarging the
elements of $\mathcal{V}$.

Consider the $Coz_{\delta}$-set
$G:=\bigcap\limits_{j\in\mathbb{N}}\bigcup\limits_{l>j}\bigcup\limits_{i\leq
l} W_{n_i,k_l}$.

Note that $A=\bigcup \{Q_{n_i}: i\in\mathbb{N}\}\subseteq G$. We
claim that $A=G$. Fix any $x\in X\setminus A$. Since $\mathcal{V}$
is a $\gamma$-cover of $X$, there is $j_0\in\mathbb{N}$ such that
$x\not\in S_{n_l,k_l}$ for all $l\geq j_0$. Find $k\in \mathbb{N}$
such that $x\not\in W_{n_i,k}$ for all $i\leq j_0$, and let $j\geq
j_0$ be such that $k_j\geq k$. Thus for all $l\geq j$ and $i\leq
l$ we have $x\not\in W_{n_i,k_l}$, because if $i\geq j_0$ we have
$x\not\in W_{n_i,k_i}\supseteq W_{n_i,k_l}$ by the choice of
$j_0$, and if $i<j_0$ we have $x\not\in W_{n_i,k}\supseteq
W_{n_i,k_l}$ by the choice of $k$. This shows that $x\not\in G$
and thus $A=G$ is a $Coz_{\delta}$-set of $X$.

Let $\mathcal{Q}=\{Q_{n_1}, Q_{n_2}, Q_{n_3}, ...\}$. Let $W\subset \mathcal{Q}$ and $W_c=\mathcal{Q}\setminus W$. Since each $Q_n$ is a zero-set, $\bigcup W_c$ is $Zer_{\sigma}$. Since $A=\bigcup \{Q_{n_i}: i\in\mathbb{N}\}=G$ is $Coz_{\delta}$, $\bigcup W=A\setminus \bigcup W_c$ is $Coz_{\delta}$. Thus, the family $\mathcal{Q}$ has the property of completely $Coz_{\delta}$-additivity. By Proposition 6, the sequence $\{\Delta_{n_i}: i\in \mathbb{N}\}$ is strongly $Coz_{\delta}$-disjoint.

\end{proof}

Recall that $C_p(X)$ is a Fr\'{e}chet-Urysohn space if and only if $X$ is a $\gamma$-space \protect\cite{GN2}.
This and Theorem \ref{th55} imply the following corollary.

\medskip

\begin{corollary}{\it Let $C_p(X)$ be a Fr\'{e}chet-Urysohn
space. Then $B_1(X)$ is Baire.}
\end{corollary}

\medskip
On the other hand, when $X$ is a Lindel\"{o}f scattered space or a Lindel\"{o}f $P$-space, $C_p(X)$ has the Fr\'{e}chet-Urysohn property (Theorems II.7.15 and II.7.16 in \cite{arh}). Thus, we get the following corollaries.

\medskip

\begin{corollary}{\it Let $X$ be a scattered Lindel\"{o}f space. Then $B_1(X)$ is Baire.}
\end{corollary}

\medskip

\begin{corollary}{\it Let $X$ be a Lindel\"{o}f $P$-space. Then $B_1(X)$ is Baire.}
\end{corollary}

\medskip

 In \cite{BG}, T. Banakh and S. Gabriyelyan introduced new class
 of space ($G^N_{\delta}$-winning) and showed that it has meager
 function space $B_1(X)$.

We say that two subsets $A,B$ of a topological space $X$ can be
separated by $G_{\delta}$-sets if there exist disjoint
$G_{\delta}$-sets $G_A, G_B\subset X$ such that $A\subseteq G_A$
and $B\subseteq G_B$.

Now we consider a topological game $G_{\delta}(X)$ played by two
players $\mathrm{S}$ and $\mathrm{N}$ (abbreviated from Separating
and Nonseparating) on a topological space $X$. The player
$\mathrm{S}$ starts the game $G_{\delta}(X)$ selecting a finite
set $S_0\subseteq X$. The player $\mathrm{N}$ responds selecting
two disjoint finite sets $A_0, B_0\subseteq X\setminus S_0$. At
the $n$-th inning the player $\mathrm{S}$ chooses a finite set
$S_n\subseteq X$ containing $S_{n-1}\cup A_{n-1}\cup B_{n-1}$ and
the player $\mathrm{N}$ responds selecting two disjoint finite
sets $A_n, B_n\subseteq X\setminus S_n$. At the end of the game
the player $\mathrm{N}$ is declared the winner if the countable
sets $A:=\bigcup_{n\in\mathbb{N}} A_n$ and
$B:=\bigcup_{n\in\mathbb{N}} B_n$ cannot be separated by
$G_{\delta}$-sets in $X$. Otherwise the player $\mathrm{S}$ wins
the game.

A topological space $X$ is defined to be

$\bullet$ {\it $G_{\delta}^N$-winning} if the player $\mathrm{N}$
has a winning strategy in the game $G_{\delta}(X)$;

$\bullet$ {\it $G_{\delta}^N$-loosing} if the player $\mathrm{N}$
has no winning strategy in the game $G_{\delta}(X)$;

$\bullet$ {\it $G_{\delta}^S$-winning} if the player $\mathrm{S}$
has a winning strategy in the game $G_{\delta}(X)$.

\medskip

Thus we have the following implications:

\medskip

$\lambda$-space $\Rightarrow$ $G_{\delta}^S$-winning $\Rightarrow$
$G_{\delta}^N$-loosing $\Leftrightarrow$ not
$G_{\delta}^N$-winning.

\medskip

Since Baireness of $B_1(X)$ implies that $X$ is
$G^N_{\delta}$-loosing  (Theorem 7.2 in \cite{BG}), we get a
positive answer to Problem 7.12 in \protect\cite{BG}: {\it Is each
metrizable $\gamma$-space $G_{\delta}^N$-loosing ?}

\medskip

\begin{corollary} {\it Each $\gamma$-space is
$G^N_{\delta}$-loosing.}
\end{corollary}

\medskip

 A set of reals $X$ is {\it concentrated} on a set $D$ if and
only if for any open set $G$ if $D\subseteq G$, then $X\setminus
G$ is countable \protect\cite{Bes}.

In 1914 Luzin constructed, using the continuum hypothesis, an
uncountable set of reals having countable intersection with every
meager set. A set of reals $X$ is a Luzin set if and only if $X$
is uncountable and concentrated on every countable dense set of
reals.

\medskip

\begin{proposition} Let $X$ be a Luzin set. Then $B_1(X)$ is meager.
\end{proposition}

\begin{proof} Assume that $B_1(X)$ is Baire. Let $A$ be a countable dense subset of $X$. By
Proposition \ref{th23}, there is $B\subset A$ such that $B$ is a
dense subset of $X$ and $B$ is a $G_{\delta}$ set in $X$. Then $X$
is countable, a contradiction.
\end{proof}

\section{Pseudocompleteness for space of Baire-one functions}

The sequence $\{\mathcal{C}_n : n\in \mathbb{N}\}$ is called {\it
pseudocomplete} if, for any family $\{U_n: n\in \mathbb{N}\}$ such
that $\overline{U_{n+1}}\subseteq U_n$ and we have $U_n\in
\mathcal{C}_n$ for each $n\in\mathbb{N}$, we have $\bigcap \{U_n:
n\in\mathbb{N}\}\neq \emptyset$. A space $X$ is called {\it
pseudocomplete} if there is a pseudocomplete sequence
$\{\mathcal{B}_n : n\in\mathbb{N}\}$ of $\pi$-bases in $X$.

It is a well-known that any pseudocomplete space is Baire and any
$\check{C}$ech-complete space is pseudocomplete. Note that if $X$
has a dense pseudocomplete subspace (in particular, if $X$ has a
dense $\check{C}$ech-complete subspace) then $X$ is pseudocomplete
(p. 47 in \protect\cite{Tk}).

\begin{lemma}\label{lem10} Let $Y$ be a topological vector space and
$L\subseteq Y$ be a dense $\check{C}$ech-complete subspace. Then the
linear span of $L$ equal to $Y$.
\end{lemma}

\begin{proof} Note that $L$ is $G_{\delta}$ in $Y$. Let $Z$ be the
linear span of $L$. If $y\in Y\setminus Z$, then $y+L$ is a dense
$\check{C}$ech-complete subspace of $Y$ (hence $G_{\delta}$)
disjoint with $L$. But this is a contradiction, since $L$ is Baire
and hence $Y$ is also Baire (by density of $L$).
\end{proof}

For $A\subseteq X$, denote by $\pi_A$ the projection  $\pi_A:
B_1(X)\rightarrow B_1(A)$, i.e., $\pi_A(f)=f\upharpoonright A$ for
$f\in B_1(X)$.

By Lemma \ref{lem10}, if $A\subseteq X$ and $\pi_A(B_1(X))$
contains a dense $\check{C}$ech-complete subspace, then
$\pi_A(B_1(X))=\mathbb{R}^A$.

\medskip

\begin{lemma}(\protect\cite{pyt1})\label{lem42} Let $Y\subseteq Z$ and $Y$ is a dense
pseudocomplete subspace of a regular space $Z$. If $\pi:
Z\rightarrow M$ is an open continuous mapping from $Z$ onto a
complete metric space $M$ then there is $M_1\subseteq \pi(Y)$ such
that $M_1$ is a dense $G_{\delta}$-set in $M$.
\end{lemma}

\medskip

\begin{theorem}\label{th77} For a space $X$ the following assertions are equivalent:

1. $B_1(X)$ is pseudocomplete.

2. Every countable subset of $X$ is strongly
$Coz_{\delta}$-disjoint.
\end{theorem}

\begin{proof}
$(1)\Rightarrow(2).$ Let $Y=B_1(X)$, $Z=\mathbb{R}^X$, $\pi=\pi_A$
where $A=\{a_i: i\in \mathbb{N}\}$ is a countable subset of $X$.
Then, by Lemma \ref{lem42}, there is $M_1\subseteq \pi_A(B_1(X))$
such that $M_1$ is a dense $\check{C}$ech-complete subspace of
$\mathbb{R}^A$. By Lemma \ref{lem10},
$\pi_A(B_1(X))=\mathbb{R}^A$. Let $h: A\rightarrow \mathbb{R}$
such that $h(a_i)=i$. Then $h$ can be extended to a Baire-one
function $g$ on $X$. The family $\{g^{-1}(i) : i\in\mathbb{N}\}$
is a completely $Coz_{\delta}$-additive system and $a_i\in
g^{-1}(i)$ for each $i\in \mathbb{N}$.

$(2)\Rightarrow(1).$ For each $n$, let $\Psi_n$ be the collection
of all basic neighborhoods of the form $\langle
f,S,\varepsilon\rangle:=\{g\in B_1(X): |g(s)-f(s)|<\varepsilon,
s\in S\}$ where $f\in B_1(X)$, $S$ is finite, and
$\varepsilon<\frac{1}{2^n}$. Note that each $\Psi_n$ is a $\pi$-base in $B_1(X)$.   Suppose, for each $n\geq 1$, $\langle
f_n,S_n,\varepsilon_n\rangle\in \Psi_n$ with $\langle
f_n,S_n,\varepsilon_n\rangle\supseteq cl(\langle
f_{n+1},S_{n+1},\varepsilon_{n+1}\rangle)$. Let $T=\bigcup \{S_n:
n\in \mathbb{N}\}$. Then $T$ is countable, and hence a strongly
$Coz_{\delta}$-disjoint set. On the set $T$ the sequence $(f_n)_n$
converges pointwise to some function $g$ on $T$. Since $T$ is
strongly $Coz_{\delta}$-disjoint, there is a pairwise disjoint
collection $\{F_t: F_t$ is a $Coz_{\delta}$ neighborhood of $t$,
$t\in T\}$  such that $\{F_t: t\in T\}$ is completely
$Coz_{\delta}$-additive.

Consider $h: \bigcup\limits_{t\in T} F_t\rightarrow \mathbb{R}$
such that $h(F_t)=g(t)$ for every $t\in T$.

By Lemma \ref{lem32}, the function $h$ can be extended to a
Baire-one function $\widetilde{h}$ on $X$. But then,
$\widetilde{h}\in \bigcap\{\langle f_n,S_n,\varepsilon_n\rangle:
n\in \mathbb{N}\}$.
\end{proof}

A family $\mathcal{F}\subseteq Y^X$ of functions from a set $X$ to
a set $Y$ is called {\it $\omega$-full} in $Y^X$ if each function
$f: Z\rightarrow Y$ defined on a countable subset $Z\subseteq X$
has an extension $\widetilde{f}\in \mathcal{F}$.

The following theorem uses notions ($G_{\delta}$-dense, countably
base-compact, strong Choquet) not defined in this paper although
we do not use them in the paper, but we recommend seeing the
definitions of these notions in \cite{BG}.

\medskip

\begin{theorem}(Theorem 2.20 in \cite{BG})\label{th42}
For any cardinal $\kappa$ and each dense subgroup $X$ of
$\mathbb{R}^{\kappa}$, the following assertions are equivalent:

(i) $X$ is $\omega$-full in $\mathbb{R}^\kappa$;

(ii) $X$ is $G_{\delta}$-dense in $\mathbb{R}^\kappa$;

(iii) $X$ is countably base-compact;

(iv) $X$ is strong Choquet;

(v) $X$ is Choquet.
\end{theorem}

Combining  Theorem \ref{th77} with Theorem \ref{th42}, we get the
following characterization of a topological space $X$ when
$B_1(X)$ is a Choquet space.

\medskip
\begin{theorem} For a space $X$ the following assertions are equivalent:

1. $B_1(X)$ is pseudocomplete;

2. $B_1(X)$ is $\omega$-full in $\mathbb{R}^X$;

3. $B_1(X)$ is $G_{\delta}$-dense in $\mathbb{R}^X$;

4. $B_1(X)$ is countably base-compact;

5. $B_1(X)$ is strong Choquet;

6. $B_1(X)$ is Choquet;

7. Every countable subset of $X$ is strongly
$Coz_{\delta}$-disjoint.

\end{theorem}

\begin{proof}
$(7)\Rightarrow(2)$. Let $Z$ be a countable subset of $X$ and $f:
Z\rightarrow \mathbb{R}$ be a function on $Z$. Since $Z$ is
strongly $Coz_{\delta}$-disjoint, there exists a pairwise disjoint
family $\alpha=\{F_z: z\in Z\}$ of $Coz_{\delta}$ subsets of $X$
such $\alpha$ forms a completely $Coz_{\delta}$-additive system.
Consider the function $h: \bigcup F_z\rightarrow \mathbb{R}$ such
that $h(x)=f(z)$ for $x\in F_z$. Since $\alpha$ forms a completely
$Coz_{\delta}$-additive system (by Lemma \ref{lem32}), there
exists $\widetilde{f}\in B_1(X)$ such that
$\widetilde{f}\upharpoonright \bigcup F_z=h$. Hence,
$\widetilde{f}\upharpoonright Z=f$.

$(2)\Rightarrow(7)$. Let $Z=\{z_i: i\in\mathbb{N}\}$ be a
countable subset of $X$. Consider the function $h:Z\rightarrow
\mathbb{R}$ such that $h(z_i)=i$ for each $i\in \mathbb{N}$. Since
$B_1(X)$ is $\omega$-full in $\mathbb{R}^X$ then there is
$\widetilde{h}\in B_1(X)$ such that $\widetilde{h}\upharpoonright
Z=h$.  Then $\{h^{-1}([i-\frac{1}{3}, i+\frac{1}{3}]): i\in
\mathbb{N}\}$ forms a disjoint completely $Coz_{\delta}$-additive
system and $z_i\in h^{-1}([i-\frac{1}{3}, i+\frac{1}{3}])$ for
each $i\in\mathbb{N}$. Thus, $Z$ is strongly
$Coz_{\delta}$-disjoint.

\end{proof}

Recall that a topological space $X$ is called a {\it
$\lambda$-space} if every countable subset is of type $G_{\delta}$
in $X$. A subset $X$ of the real line $\mathbb{R}$ is called a {\it $\lambda$-set} if each countable subset $A\subset X$ is $G_{\delta}$ in $\mathbb{R}$.

\medskip
\begin{corollary}\label{cor45} {\it Let $X$ be a space of countable pseudocharacter. A space $B_1(X)$ is pseudocomplete if and only if $X$ is a
 $\lambda$-space.}
\end{corollary}

\medskip
Let us recall the definition of some small uncountable cardinal
(see \protect\cite{handbook}, p.149):

 $\mathfrak{b}=\min\{|X|:$ $X$ has a countable
pseudocharacter but $X$ is not a $\lambda$-space$\}$.

Note that the small cardinal $\mathfrak{b}$ has a different
standard definition (see \protect\cite{handbook}, p.115), the
given formula is one of the equivalent descriptions.

\medskip

\begin{proposition} Let $X$ be a space $X$ of countable pseudocharacter. If $|X|<\mathfrak{b}$ then $B_1(X)$ is Choquet (pseudocomplete).
\end{proposition}

\begin{proof} By definition of $\mathfrak{b}$, if $|X|<\mathfrak{b}$ then $X$ is a
$\lambda$-space.

By Theorem 1.3 in \protect\cite{BG}, for any space $X$ of countable pseudocharacter, $B_1(X)$ is Choquet if and only if $X$ is a $\lambda$-space.
\end{proof}

Recall that a space $X$ is called a {\it $Q$-space} if every
subset is an $F_{\sigma}$-set in $X$.

By Theorem 2.1 in \protect\cite{Os4}, Lemma \ref{lem32} and
Theorem \ref{th77}, we have the following result.

\medskip
\begin{corollary}\label{cor46} {\it Let $X$ be a space of countable pseudocharacter. A space $B_1(X)$ is
pseudocomplete and realcompact if and only if $X$ is a
$Q$-space.}
\end{corollary}

\medskip

The following example we consider under the model of set theory in
which every $\mathfrak{b}$-scale set is a $\gamma$-space
(Corollary 1.5 in \cite{CRZ}).

\bigskip

\bigskip

\begin{example} It is consistent with $ZFC$ there is a zero-dimensional metrizable
separable space $X$ such that

1. $B_1(X)$ is not pseudocomplete;

2. $B_1(X)$ is Baire;

3. $|X|=\mathfrak{b}$;

4. $X$ is a $\gamma$-space;

5. $X$ is not a $\lambda$-space.

\end{example}

\begin{proof} Let $X$ be the space $H$ from Theorem 10 of
\protect\cite{BT} or $X$ from Example 8.4 of \protect\cite{BG}.
Let us recall its construction.

Given two functions $f,g: \omega\rightarrow \omega$, we write
$f\leq^*g$ ($f\leq g$) if the set $\{n\in \mathbb{N}: f(n)\not\leq
g(n)\}$ is finite (empty). A subset $B\subseteq \omega^{\omega}$
is {\it unbounded} if for any $y\in \omega^{\omega}$ there is
$x\in B$ such that $x\not\leq^* y$. The cardinal $\mathfrak{b}$
can be equivalently defined as the smallest cardinality of an
unbounded subset of $\omega^{\omega}$.

By $\omega^{\uparrow \omega}$ we denote the family of increasing
functions from $\omega$ to $\omega$. There exists a transfinite
sequence $\{f_{\alpha}\}_{\alpha<\mathfrak{b}}$ of increasing
functions $f_{\alpha}: \omega\rightarrow \omega$ such that the set
$S=\{f_{\alpha}\}_{\alpha<\mathfrak{b}}$ is unbounded in
$\omega^{\omega}$ and $f_{\alpha}\leq^*f_{\beta}$ for every
$\alpha<\beta$ in $\mathfrak{b}$ (see Example 8.4 of
\protect\cite{BG}).

Consider the closure $\overline{\omega}=[0,\omega]$ of $\omega$ in
the ordinal $\omega+1$ endowed with the order topology. Let
$C\subset \overline{\omega}^{\omega}$ be the countable set of
functions $f:\omega\rightarrow\overline{\omega}$ such that there
exists $n\in \mathbb{N}$ so that $f(i)<f(j)$ for $i<j<n$ and
$f(k)=\omega$ for all $k\geq n$. Note that the subspace
$\overline{\omega}^{\uparrow \omega}:=\omega^{\uparrow \omega}\cup
C$ in $\overline{\omega}^{\omega}$ is closed and hence compact.
The subspace $X:=S\cup C$ of $\overline{\omega}^{\omega}$ has the
property (1)-(5).

$\bullet$ $|X|=\mathfrak{b}$.

$\bullet$ By the proof (ii) in Example 8.4 of \protect\cite{BG},
the set $C$ is not a $G_{\delta}$-set in $X$, hence, $X$ is not a
$\lambda$-space. By Corollary \ref{cor45}, $B_1(X)$ is not
pseudocomplete.

$\bullet$ Galvin and Miller \protect\cite{GM} constructed under
$\mathfrak{p}=\mathfrak{c}$ a $\mathfrak{b}$-scale set which is a
$\gamma$-space. According to \protect\cite{CRZ}, it is consistent
that for any $\mathfrak{b}$-scale
$S=\{f_{\alpha}\}_{\alpha<\mathfrak{b}}$ the space $X=S\cup C$ is
a $\gamma$-space and, hence, by Theorem \ref{th55}, $B_1(X)$ is
Baire.

\end{proof}

Let us give several examples of subsets $X$ of the real line $\mathbb{R}$ for which $B_1(X)$ is Baire.

\medskip
{\bf ZFC Examples}

\medskip
(1) $B_1(\mathbb{Q})$, where $\mathbb{Q}$ is the space of all rational numbers (or any countable subset of $\mathbb{R}$), is Baire. This is because $\mathfrak{b}$ is uncountable.

\medskip
(2) Let $X$ be an uncountable $\lambda$-set that is a subset of the real line. Then $B_1(X)$ is Baire.

\medskip
{\bf Consistent Examples}

\medskip

(3) For any uncountable subset $X$ of the real line of cardinality $<\mathfrak{b}$, $B_1(X)$ is Baire.

\medskip

(4) Let $X$ be an uncountable $\gamma$-set that is a subset of the real line.  Then $B_1(X)$ is Baire.

\medskip

{\bf Remark.} There is a $ZFC$ example of an uncountable subset of the real line that is a $\lambda$-space (see \cite{BG} page 3). It is consistent with $ZFC$ that $\omega_1<\mathfrak{b}$. When this is the case, any subset of the real line of cardinality $\omega_1$ would be a $\lambda$-set.

According to Gerlits and Nagy \protect\cite{GN}, under $MA$+ not $CH$, uncountable $\gamma$-sets exist since any subset of the real line of cardinality less than continuum is a $\gamma$-set. On the other hand, in Laver's model for the Borel conjecture, all $\gamma$-sets are countable \cite{Lav}.

\medskip
{\bf Comment.} All of these examples (1)-(4) are examples of Baire $B_1(X)$ such that $C_p(X)$ is not Baire. Note that there exists a non-trivial convergent sequence in each of the $X$. The example of $B_1(\mathbb{Q})$ is a clear illustration of the difference between strong discrete sequence and strongly $Coz_{\delta}$-disjoint sequence.

\medskip
There is a $ZFC$ example for Baire $C_p(X)$ that is not pseudocomplete (Ex. 469 in \cite{Tk}). Example 1 is an example of a Baire $B_1(X)$ that is not pseudocomplete. This a consistent example since it involves an uncountable $\gamma$-set.
\medskip

{\bf Question 1.} Is there a $ZFC$ example for a Baire $B_1(X)$ that is not pseudocomplete?

\section{Countable dense homogeneity of space of Baire-one functions}

A space $X$ is {\it countable dense homogeneous} (CDH) if $X$ is
separable and given countable dense subsets $D,D'\subseteq X$,
there is a homeomorphism $h: X\rightarrow X$ such that $h[D]=D'$.
Canonical examples of CDH spaces include the Cantor set, Hilbert
cube, space of irrationals, and all separable complete metric
linear spaces and manifolds modeled on them. An easy example of a
non-metrizable CDH space is the Sorgenfrey line.

In \cite{dkm}, it is proved that every CDH topological vector
space is a Baire space.

If $X$ is a space and $A\subseteq X$, then the sequential closure
of $A$, denote by $[A]_{seq}$, is the set of all limits of
sequences from $A$. A set $D\subseteq X$ is said to be
sequentially dense if $X=[D]_{seq}$.

Note that if $B_1(X)$ is CDH, then it is strongly sequentially
separable, i.e. $B_1(X)$ is separable and every countable dense
subset of $B_1(X)$ is sequentially dense.

Many topological properties are defined or characterized in terms
of the following selection principle.  Let $\mathcal{A}$ and
$\mathcal{B}$ be sets consisting of families of subsets of an
infinite set $X$. Then:

$S_1(\mathcal{A},\mathcal{B})$ is the selection hypothesis: for
each sequence $(A_n: n\in\mathbb{N})$ of elements of $\mathcal{A}$
there is a sequence $(b_n: n\in\mathbb{N})$ such that for each
$n$, $b_n\in A_n$, and $\{b_n: n\in\mathbb{N}\}$ is an element of
$\mathcal{B}$.

For a topological space $X$ we denote:

$\bullet$ $\mathcal{B}_{\Omega}$ -- the family of countable Baire
$\omega$-covers of $X$;

$\bullet$ $\mathcal{B}_{\Gamma}$ -- the family of countable Baire
$\gamma$-covers of $X$.

\medskip

Then we have the following results.

$\bullet$ If $B_1(X)$ is strongly sequentially separable, then $X$
has the property $S_1(\mathcal{B}_{\Omega}, \mathcal{B}_{\Gamma})$
\protect\cite{os3}.

$\bullet$ If $X$ has the property $S_1(\mathcal{B}_{\Omega},
\mathcal{B}_{\Gamma})$ then it is a $\sigma$-space, i.e. each
$G_{\delta}$-subset of $X$ is an $F_{\sigma}$-subset
\protect\cite{scbo}.

$\bullet$ If $B_1(X)$ is separable, then $X$ is submetrizable
\protect\cite{os2}.

\medskip

Thus, we get that if $B_1(X)$ is CDH then  $B_1(X)$ is Baire and a
submetrizable space $X$ has the property
$S_1(\mathcal{B}_{\Omega}, \mathcal{B}_{\Gamma})$ (hence, is a
$\sigma$-space).

\medskip

Let us recall the definitions of some small uncountable cardinals:

$\bullet$ $\mathfrak{q}_0:=\min \{|X|: X\subseteq \mathbb{R}, X$
is not a $Q$-space $\};$

$\bullet$ $\mathfrak{q}:=\min\{\kappa:$ if $X\subseteq \mathbb{R}$
and $|X|\geq \kappa$, then $X$ is not a $Q$-space $\};$

$\bullet$ $\mathfrak{b}:=\min\{|X|: X$ is perfect but is not a
$\sigma$-set $\};$

$\bullet$ $\mathfrak{p}$ is the smallest cardinality of a family
$\mathcal{F}$ of infinite subsets of $\omega$ such that every
finite subfamily $\mathcal{E}\subseteq \mathcal{F}$ has infinite
intersection $\bigcap \mathcal{E}$ and for any infinite set
$I\subseteq \omega$ there exists a set $F\in \mathcal{F}$ such
that $I\setminus F$ is infinite.
\medskip

It is known (see \protect\cite{bmz}) that $\mathfrak{p}\leq
\mathfrak{q}_0 \leq \min \{\mathfrak{b},\mathfrak{q}\}\leq
\mathfrak{q}\leq \mathfrak{c}$.

\medskip
It is consistent that there are no uncountable $\sigma$-space
(Theorem 22 in \protect\cite{mil}).

\medskip

\begin{proposition} It is consistent that there are no uncountable
separable metrizable space $X$ such that $B_1(X)$ is CDH.
\end{proposition}

\medskip

On the other hand, if $|X|<\mathfrak{q}_0$ then $X$ is a
$Q$-space. For a perfect normal space $X$, if $|X|<\mathfrak{q}_0$
then $B_1(X)=\mathbb{R}^X$. But, $\mathbb{R}^X$ is CDH if, and
only if, $|X|<\mathfrak{p}$ \protect\cite{stzh}. Since
$\mathfrak{p}\leq \mathfrak{q}_0$, we have the following result.

\medskip

\begin{proposition} Let $X$ be a perfect normal space and
$|X|<\mathfrak{p}$. Then $B_1(X)$ is CDH.
\end{proposition}

\medskip

In particular, if $X$ is a countable metrizable space, then
$B_1(X)$ is CDH.

\medskip
\begin{proposition}$(\mathfrak{p}<\mathfrak{q}_0)$. It is consistent that there are a set of reals $X$ such that $B_1(X)$ is Baire, but is not
CDH.
\end{proposition}

\begin{proof} Let $X$ be a set of real of cardinality
$\mathfrak{p}$. Since $\mathfrak{p}<\mathfrak{q}_0$, $X$ is a
$Q$-set and, hence, $B_1(X)=\mathbb{R}^X$. But, $\mathbb{R}^X$ is
CDH if and only if $|X|<\mathfrak{p}$ \protect\cite{stzh}. Hence,
$B_1(X)$ is not CDH. Since
$\mathfrak{p}<\mathfrak{q}_0\leq\mathfrak{b}$, $X$ is a
$\sigma$-set (hence, $\lambda$-space, i.e. every countable subset
is $G_{\delta}$ in $X$) and $B_1(X)$ is Baire (see
\protect\cite{BG}).
\end{proof}

\begin{proposition}
If an uncountable set of real numbers is concentrated on a
countable subset of itself, then $B_1(X)$ is not CDH.
\end{proposition}

\begin{proof}
If an uncountable set of real numbers is concentrated on a
countable subset of itself, then it does not have property
$S_1(\mathcal{B}_{\Omega}, \mathcal{B}_{\Gamma})$ (see Corollary 5
in \protect\cite{scbo}) and hence $B_1(X)$ is not CDH.
\end{proof}

\medskip

{\bf Acknowledgements.} The author would like to thank O.F.K.~Kalenda and
S.~Gabriyelyan for several valuable comments.

The author would like to thank the
referee for careful reading and valuable comments.

\bibliographystyle{model1a-num-names}
%\bibliography{sn-bibliography}

\begin{thebibliography}{30}

\bibitem{arh}
A.V. Arhangel'skii, Topological function spaces, Moskow. Gos.
Univ., Moscow, (1989), 223 pp. (Arhangel'skii A.V., Topological
function spaces, Kluwer Academic Publishers, Mathematics and its
Applications, 78, Dordrecht, 1992 (translated from Russian)).



\bibitem{BG}
T. Banakh, S. Gabriyelyan, Baire category of some Baire type
function spaces, Topology and its Applications, {\bf 272} (2020),
107078.

\bibitem{bmz}
T. Banakh, M. Machura, L. Zdomskyy, On critical cardinalities
related to $Q$-sets, Math. Bull. T. Shevchenko Sci. Soc. {\bf 11}
(2014), 21--32.

\bibitem{BT}
T. Bartoszy\'{n}ski, B. Tsaban, Hereditary topological
diagonalizations and the Menger-Hurewicz conjectures, Proc. Amer.
Math. Soc. {\bf 132}:2 (2006), 605--615.

\bibitem{Bes}
A.S. Besicovich, Concentrated and rarified sets of points, Acta
Math., {\bf 62} (1933), 289--300.



\bibitem{CRZ}
D. Chodounsk\'{y}, D. Repov\v{s}, L. Zdomskyy, Mathias forcing and
combinatorial covering properties of filters, J. Symb. Log. {\bf
80}:4 (2015), 1398--1410.






\bibitem{Ch}
M.M. Choban, Functionally countable spaces and Baire functions,
Serdica Math. J., {\bf 23} (1997), 233--242.



\bibitem{dkm}
T. Dobrowolski, M. Krupski, W. Marciszewski, A countable dense
homogeneous topological vector space is a Baire space, Proc. Amer.
Math. Soc., {\bf 149} (2021), 1773--1789.


\bibitem{handbook}
E. K. van Douwen, The Integers and Topology, in {\it Handbook of
Set-Theoretic Topology}, K. Kunen and J. Vaughan (eds.),
North-Holland (1984), 111--168.


\bibitem{Eng}
R. Engelking, General Topology, Revised and completed edition,
Heldermann Verlag Berlin (1989).

\bibitem{KS}
O.F.K. Kalenda, J. Spurn\'{y}, Extending Baire-one functions on
topological spaces, Topol. Appl., {\bf 149} (2005), 195--216.



\bibitem{Lav}
R. Laver, On the consistency of Borel's conjecture, Acta Math, {\bf 137} (1976), 151--169.


\bibitem{LM}
D. Lutzer, R. McCoy, Category in function spaces. I, Pacific J.
Math. {\bf 90} (1980), no.1, 145--168.


\bibitem{lmz1}
J. Luke\v{s}, J. Mal\'{y}, L. Zaji\v{c}ek, Fine Topology Methods
in Real Analysis and Potential Theory, Lecture Notes in
Mathematics, vol. 1189, Springer, Berlin, 1986.



\bibitem{mil}
A.W. Miller, On the length of Borel hierarchies, Annals of
Mathematical Logic, {\bf 16} (1979), 233--267.



\bibitem{vD}
J.van Mill ed, Eric K. van Douwen, Collected Papers, vol. 1,
North-Holland, Amsterdam (1994).



\bibitem{GO}
S. Gabriyelyan, A.V. Osipov, Topological properties of some
function spaces, Topology and its Applications, {\bf 279} (2020),
107248.




\bibitem{GM}
F. Galvin, A.W. Miller, $\gamma$-sets and other singular sets of
real numbers, Topology Appl. {\bf 17} (1984), 145--155.



\bibitem{GN}
J. Gerlits, Zs. Nagy, Some properties of $C(X)$, I, Topology Appl.
{\bf 14}:2 (1982), 151--161.

\bibitem{GN2}
J. Gerlits, Zs. Nagy, Some properties of $C(X)$, II, Topology
Appl. {\bf 15}:3 (1983), 255--262.












\bibitem{GMB}
G. Gruenhage, D.K. Ma, Baireness of $C_k(X)$ for locally compact
$X$, Topol. Appl., {\bf 80} (1997), 131--139.

\bibitem{Is}
T. Ishii, The Tychonoff functor and related topics, in: Topics in
General Topology, K. Morita and J. Nagata, eds., North- Holland
(1989) 203--243.

\bibitem{Ok}
S. Oka, The Tychonoff functor and product spaces, Proc. Japan
Acad. 54 (1978) 97--100.


\bibitem{os2}
A.V. Osipov, On generalization of theorems of Pestryakov, Topol.
Appl., {\bf 283} (2020), 107399.


\bibitem{os3}
A.V. Osipov, Classification of selectors for sequences of dense
sets of Baire functions, Topol. Appl., {\bf 258} (2019), 251--267.


\bibitem{Os4}
A.V. Osipov, On the Hewitt realcompactification and
$\tau$-placedness of function spaces, Trudy Instituta Matematiki i
Mekhaniki URO RAN, {\bf 25}:4 (2019), 177--183 (in Russian).


\bibitem{Ox}
J.C. Oxtoby, The Banach-Mazur game and Banach Category Theorem,
in: Contributions to the theory of games, Vol. III, Annals.
Studies {\bf 39} (1957), 159--163.

\bibitem{Ox1}
J.C. Oxtoby, Cartezian products of Baire spaces, Fund. Math. {\bf
49}:2 (1961), 157--166.

\bibitem{pyt1}
E.G. Pytkeev, Baire property of spaces of continuous functions.
Mathematical Notes of the Academy of Science of the USSR {\bf 38}
(1985), 908--915.(Translated from Matematicheskie Zametki, {\bf
38}:5, 726--740).

\bibitem{TKA}
V.V. Tkachuk, Characterization of the Baire property in $C_p(X)$
by the properties of the space $X$, Researh papers, Topology-Maps
and Extensions of Topological Spaces (Ustinov)(1985), 21--27.

\bibitem{Tk}
V.V. Tkachuk, A $C_p$-Theory Problem Book. Topological and
Function spaces., Springer, 2011.




\bibitem{Vid}
G. Vidossich, On topological spaces whose functions space is of
second category, Invent.Math., {\bf 8}:2 (1969), 111--113.

\bibitem{stzh}
J. Stepr\"{a}ns, H.X. Zhou, Some results on CDH spaces - I, Topol.
Appl., {\bf 28} (1988), 147--154.


\bibitem{scbo}
M. Scheepers, B. Tsaban, The combinatorics of Borel covers, Topol.
Appl., {\bf 121} (2002), 357--382.





\bibitem{Whi}
H.E. White, Jr., Topological spaces that are $\alpha$-favorable
for a player with perfect information, Proc. Amer. Math. Soc.,
{\bf 50} (1975), 477--482.


\end{thebibliography}

\end{document}